\documentclass[10pt,reqno,oneside]{amsart}

\usepackage{amsmath,amssymb,amsthm}
\usepackage{color,cite}
\usepackage{marginnote}

\usepackage{mathrsfs}
\usepackage{enumerate}
\usepackage{verbatim}

\usepackage{stmaryrd} 

\theoremstyle{plain}
\newtheorem{thm}{Theorem}[section]
\newtheorem*{thm*}{Theorem}
\newtheorem{cor}[thm]{Corollary}
\newtheorem*{cor*}{Corollary}
\newtheorem{prop}[thm]{Proposition}
\newtheorem*{prop*}{Proposition}
\newtheorem{lem}[thm]{Lemma}
\newtheorem*{lem*}{Lemma}

\newtheorem*{claim*}{Claim}

\theoremstyle{definition}
\newtheorem{defn}[thm]{Definition}
\newtheorem*{defn*}{Definition}

\theoremstyle{remark}

\newtheorem*{rmk*}{Remark}

\newtheorem*{exm*}{Example}

\numberwithin{equation}{section}

\newcommand{\N}{\mathbb{N}}

\newcommand{\R}{\mathbb{R}}

\newcommand{\Z}{\mathbb{Z}}
\newcommand{\bP}{\mathbb{P}}

     \newcommand{\cC}{\mathcal C}
     \newcommand{\cD}{\mathcal D}
     
     \newcommand{\cF}{\mathcal F}

     \newcommand{\cK}{\mathcal K}

     \newcommand{\cM}{\mathcal M}

     \newcommand{\cS}{\mathcal S}

\newcommand{\p}{\partial}
\newcommand{\grad}{\nabla}
\newcommand{\la}{\langle}
\newcommand{\ra}{\rangle}
\let\div\relax
\DeclareMathOperator{\div}{div}
\DeclareMathOperator{\curl}{curl}
\newcommand{\loc}{{\rm loc}}
\newcommand{\supp}{{\rm supp}}

\renewcommand{\:}{\colon}
\newcommand{\norm}[1]{\lVert #1 \rVert}

\newcommand{\wto}{\rightharpoonup} 
\newcommand{\wstar}{\overset{\ast}{\rightharpoonup}}
\newcommand{\ind}{\mathbf{1}} 
\newcommand{\inds}[1]{\mathbf{1}_{\{ #1 \}}} 
\newcommand{\into}{\hookrightarrow}
\newcommand{\upto}{\uparrow}
\newcommand{\dto}{\downarrow}

\title[Blow-up criteria for NSE in critical Besov spaces]{Blow-up criteria for the Navier-Stokes equations in non-endpoint critical Besov spaces}
\author{Dallas Albritton$^\ast$}
\thanks{$^\ast$University of Minnesota}
\date\today

\begin{document}

\begin{abstract}
	We obtain an improved blow-up criterion for solutions of the Navier-Stokes equations in critical Besov spaces. If a mild solution $u$ has maximal existence time $T^* < \infty$, then the non-endpoint critical Besov norms must become infinite at the blow-up time:
\begin{equation}\notag
\lim_{t \upto T^*} \norm{u(\cdot,t)}_{\dot B^{-1+3/p}_{p,q}(\R^3)} = \infty, \quad 3 < p,q < \infty.
\end{equation}
In particular, we introduce \emph{a priori} estimates for the solution based on elementary splittings of initial data in critical Besov spaces and energy methods. These estimates allow us to rescale around a potential singularity and apply backward uniqueness arguments. The proof does not use profile decomposition. 
\end{abstract}
\maketitle

\section{Introduction}

We are interested in blow-up criteria for solutions of the incompressible Navier-Stokes equations
\begin{equation}\tag{NSE}
\begin{aligned}
	\p_t u - \Delta u + u \cdot \nabla u + \nabla p &= 0 \\
	\div u &= 0 \\
	u(\cdot,0) &= u_0
\end{aligned}
\end{equation}
in $Q_T := \R^3 \times (0,T)$ with divergence-free initial data $u_0 \in C^\infty_0(\R^3)$. It has been known since Leray \cite{leray} that a unique smooth solution with sufficient decay at infinity exists locally in time. Furthermore, Leray proved that there exists a constant $c_p > 0$ with the property that if $T^* < \infty$ is the maximal time of existence of a smooth solution, then
\begin{equation}\label{subcrit}
\norm{u(\cdot,t)}_{L^p(\R^3)} \geq c_p \left( \frac{1}{\sqrt{T^*-t}} \right)^{1-3/p}
\end{equation}
for all $3 < p \leq \infty$.
Such a characterization exists because the Lebesgue norms in this range are subcritical with respect to the natural scaling symmetry of the Navier-Stokes equations,
\begin{equation}\label{scaling}
u(x,t) \to \lambda u(\lambda x, \lambda^2 t), \quad p(x,t) \to \lambda^2 p(\lambda x, \lambda^2 t).
\end{equation}

The behavior of the critical $L^3$
norm near a potential blow-up was unknown until the work of Escauriaza, Seregin, and {\u S}ver{\'a}k \cite{sverak03}, who discovered an endpoint local regularity criterion in the spirit of the classical work by Serrin \cite{serrin}. In particular, they demonstrated that if $u$ is a Leray-Hopf solution of the Navier-Stokes equations with maximal existence time $T^* <\infty$, then
\begin{equation}\label{l3criterion}
\limsup_{t \upto T^*} \, \norm{u(\cdot,t)}_{L^3(\R^3)} = \infty.
\end{equation}
Their proof uses the $\varepsilon$-regularity criterion of Caffarelli, Kohn, and Nirenberg \cite{CKN} in an essential way, and moreover it introduced powerful backward uniqueness arguments for studying potential singularities of solutions to the Navier-Stokes equations.
The proof is by contradiction:
If a solution forms a singularity but remains in the critical space $L^\infty_t L^3_x(Q_{T^*})$, then one may zoom in on the singularity using the scaling symmetry and obtain a weak limit. The limit solution will form a singularity but also vanish identically at the blow-up time. By backwards uniqueness, the limit solution $u$ must be identically zero in space-time, which contradicts that it forms a singularity. This method was adapted by Phuc \cite{phuc} to cover blow-up criteria in Lorentz spaces. Interestingly, backwards uniqueness techniques have also been employed in the context of harmonic map heat flow by Wang \cite{wang}. For a different proof of the criterion, see \cite{du}.

A few years ago, Seregin \cite{sereginl3} improved the blow-up criterion of Escauriaza-Seregin-{\u S}ver{\'a}k by demonstrating that the $L^3$ norm must become infinite at a potential blow-up:
 \begin{equation}\label{l3seregin}
 \lim_{t \upto T^*} \norm{u(\cdot,t)}_{L^3(\R^3)} = \infty.
 \end{equation}
 The main new difficulty in the proof is that one no longer controls the $L^\infty_t L^3_x$ norm when zooming in on a potential singularity. Seregin addressed this difficulty by relying on certain properties of the local energy solutions introduced by Lemari{\'e}-Rieusset \cite{lemarie1,kikuchi}. However, an analogous theory of local energy solutions was not known in the half space $\R^3_+:= \{ x \in \R^3 : x_3 > 0 \}$.\footnote{It appears that this theory has recently been developed in \cite{localenergynew}.} In order to overcome this obstacle, Barker and Seregin \cite{sereginhalf} introduced new \emph{a priori} estimates which depend only on the norm of the initial data in the Lorentz spaces $L^{3,q}$, $3 < q < \infty$. This is accomplished by splitting the solution as
\begin{equation}\label{introsplitting}
u = e^{t \Delta} u_0 + w,
\end{equation}
where $w$ is a correction in the energy space. The new estimates allowed Barker and Seregin to obtain an analogous blow-up criterion for Lorentz norms in the half space. Later, Seregin and {\u S}ver{\'a}k abstracted this splitting argument into the notion of a global weak $L^3$-solution \cite{sverakl3}. We direct the reader to the paper \cite{sereginweakl3} for global weak solutions with initial data in $L^{3,\infty}$.

Recently, there has been interest in adapting the ``concentration compactness + rigidity'' roadmap of Kenig and Merle \cite{kenigmerle}
to blow-up criteria for the Navier-Stokes equations. This line of thought was initiated by Kenig and G. Koch in \cite{koch3} and advanced to its current state by Gallagher, Koch, and Planchon in \cite{koch2,koch}. Gallagher et al. succeeded in extending a version of the blow-up criterion to the negative regularity critical Besov spaces $\dot B^{s_p}_{p,q}(\R^3)$, $3 < p,q < \infty$. Here, $s_p := -1+3/p$ is the critical exponent. Specifically, it is proved in \cite{koch} that if $T^* < \infty$, then
\begin{equation}\label{besovcriterion}
	\limsup_{t \upto T^*} \, \norm{u(\cdot,t)}_{\dot B^{s_p}_{p,q}(\R^3)} = \infty.
\end{equation}
This proof is also by contradiction. If there is a blow-up solution in the space $L^\infty_t (\dot B^{s_p}_{p,q})_x$, then one may prove via profile decomposition that there is a blow-up solution in the same space and with minimal norm (made possible by small-data-global-existence results \cite{kato}). This solution is known as a critical element. By essence of its minimality, the critical element vanishes identically at the blow-up time, so one may apply the backward uniqueness arguments of Escauriaza, Seregin, and {\u S}ver{\'a}k to obtain a contradiction. The main difficulty lies in proving the existence of a profile decomposition in Besov spaces, which requires some techniques from the theory of wavelets \cite{kochbesov,bahouri}. A secondary difficulty is obtaining the necessary estimates near the blow-up time in order to apply the $\varepsilon$-regularity criterion. We note that the paper \cite{koch3} appears to be the first application of Kenig and Merle's roadmap to a parabolic equation. The nonlinear profile decomposition for the Navier-Stokes equations was first proved by Gallagher in \cite{gallagherprof}. The paper \cite{chemin} contains further interesting applications of profile decomposition techniques to the Navier-Stokes equations. \\

In this paper, we obtain the following improved blow-up criterion for the Navier-Stokes equations in critical spaces:
\begin{thm}[Blow-up criterion]\label{main}
	Let $3 < p,q < \infty$ and $u_0 \in \dot B^{s_p}_{p,q}(\R^3)$ be a
	divergence-free vector field. Suppose $u$ is the mild solution of the Navier-Stokes equations on
	$\R^3 \times [0,T^*)$ with initial data $u_0$ and maximal time of
		existence $T^*(u_0)$. If $T^* < \infty$, then
		\begin{equation}\label{besovblowupcrit}
\lim_{t \upto T^*} \norm{u(\cdot,t)}_{\dot B^{s_p}_{p,q}(\R^3)} = \infty.
\end{equation}
\end{thm}

The local well-posedness of mild solutions, including characterizations of the maximal time of existence, are reviewed in Theorem \ref{mildexist}.

Let us discuss the novelty of Theorem \ref{main}. First, this theorem extends Seregin's $L^3$ criterion \eqref{l3seregin} to the scale of Besov spaces and replaces the $\limsup$ condition in Gallagher-Koch-Planchon's criterion \eqref{besovcriterion}. Our proof does not rely on the profile decomposition techniques in the work of Gallagher et al. \cite{koch} and may be considered to be more elementary. Rather, our methods are based on the rescaling procedure in Seregin's work \cite{sereginl3}. Regarding optimality, it is not clear whether Theorem \ref{main} is valid for the endpoint spaces $\dot B^{s_p}_{p,\infty}$ and $BMO^{-1}$, which contain non-trivial $-1$-homogeneous functions, e.g., $|x|^{-1}$. Indeed, if the blow-up profile $u(\cdot,T^*)$ is locally a scale-invariant function, then rescaling around the singularity no longer provides useful information.\footnote{Since the submission of this paper, T. Barker has proven the blow-up criterion $\lim_{t \upto T^*} \norm{u(\cdot,t)}_{\dot B^{s_p}_{p,\infty}} = \infty$, using Calder{\'o}n-type solutions, under the extra assumption that $u(\cdot,T^*)$ vanishes in the rescaling limit \cite{barkernewest}. See also the forthcoming work \cite{albrittonbarkernew} of T. Barker and the author.} It is likely that this is an essential issue and not merely an artifact of the techniques used here. For instance, one may speculate that if Type I blow-up occurs (in the sense that the solution blows up in $L^\infty$ at the self-similar rate), then the $VMO^{-1}$ norm does not blow-up at the first singular time.


As in previous works on blow-up criteria for the Navier-Stokes equations, the main difficulty we encounter is in obtaining \emph{a priori} estimates for solutions up to the potential blow-up time. We also require that the estimates depend only on the norm of the initial data in $\dot B^{s_p}_{p,q}$. The low regularity of this space creates a new difficulty because the splitting \eqref{introsplitting} does not appear to work in the space $\dot B^{s_p}_{p,q}$ when $2/q + 3/p < 1$. One problem is that when obtaining energy estimates for the correction $w$ in \eqref{introsplitting}, the operator
\begin{equation}\label{problem}
	(U,u_0) \mapsto \int_0^T \int_{\R^3} e^{t \Delta} u_0 \cdot \nabla U \cdot U \, dx \, dt
\end{equation}
is not known to be bounded for $U \in L^\infty_t L^2_x \cap L^2_t \dot H^1_x$ and $u_0 \in \dot B^{s_p}_{p,q}$. This is because $e^{t \Delta} u_0$ ``just misses'' the critical Lebesgue space $L^r_t L^p_x$ with $2/r + 3/p = 1$. Therefore, to obtain the necessary a priori estimates, we rely on a method essentially established by C.P. Calder{\'o}n \cite{calderon}. The idea is as follows. We split the critical initial data $u_0 \in \dot B^{s_p}_{p,p}$ into supercritical and subcritical parts:
\begin{equation}\label{oursplittingintro}
	u_0 = U_0 + V_0 \in L^2 + \dot B^{s_q+\varepsilon}_{q,q}.
\end{equation}
When small, the data $V_0$ in a subcritical Besov space gives rise to a mild solution $V$ on a prescribed time interval (not necessarily a global mild solution). The supercritical data $U_0 \in L^2$ serves as initial data for a correction $U$ in the energy space. We will refer to solutions which split in this way as Calder{\'o}n solutions, see Definition \ref{calderonsol}, and we construct them in the sequel. Note that the unboundedness of \eqref{problem} is similarly problematic when proving weak-strong uniqueness in Besov spaces. In recent work on weak-strong uniqueness, Barker \cite{barkernew} has also dealt with this issue via the splitting \eqref{oursplittingintro}.
We remark that Calder{\'o}n's original idea was to construct global weak solutions by splitting $L^p$ initial data for $2 < p < 3$ into small data in $L^3$ and a correction in $L^2$. This idea has also been used to prove the stability of global mild solutions \cite{gallagher,auscher}.

Let us briefly constrast the solutions we construct via \eqref{oursplittingintro} to the global weak $L^3$ solutions introduced by Seregin and {\u S}ver{\'a}k in \cite{sverakl3}, which are constructed via the splitting \eqref{introsplitting}. The correction term $w$ in \eqref{introsplitting} has zero initial data, which allows one to prove that an appropriate limit of solutions also satisfies the energy inequality up to the initial time. For this reason, the global weak $L^3$-solutions are continuous with respect to weak convergence of initial data in $L^3$. Since the splitting \eqref{oursplittingintro} requires the correction to have non-zero initial condition $U_0$, an analogous continuity result is not as obvious for Calder{\'o}n solutions. We do not seek to prove such a result here
as to avoid burdening the paper technically, but we expect that it may be shown by adapting various ideas in \cite{sereginissue,barkernew}. Using similar ideas, we expect that one could prove that all Calder{\'o}n solutions agree with the mild solution on a short time interval.


Here is the layout of the paper:
\begin{itemize}
	\item In Section 2, we prove the existence of Calder{\'o}n solutions that agree with the mild solution until the blow-up time. This is the content of Theorems \ref{existence} and \ref{mildcalderon}. We also describe the properties of weak limits of Calder{\'o}n solutions in Theorem \ref{weakconv}. The splitting arguments for initial data in Besov spaces are contained in Lemma \ref{splitdata}.
	\item In Section 3, we prove Theorem \ref{main} using the results of Section 2.
	\item Section 4 is an extensive appendix that summarizes the local well-posedness theory of mild solutions in homogeneous Besov spaces and collects well-known theorems about $\varepsilon$-regularity and backward uniqueness. We include it for the reader's convenience and to make the paper self-contained.
\end{itemize}
\noindent Notation is reviewed in the appendix. One important point is that we do not distinguish the notation of scalar-valued and vector-valued functions.

After completion of the present work, we learned that T. Barker and G. Koch \cite{barkerpersonal} obtained a different proof of the blow-up criterion \eqref{besovblowupcrit}. Their proof treats mild solutions by exploiting certain properties of the local energy solutions of Lemari{\'e}-Rieusset.

\subsection*{Acknowledgments}

The author would like to thank his advisor, Vladim{\'i}r {\u S}ver{\'a}k, as well as Daniel Spirn, Alex Gutierrez, Laurel Ohm, and Tobias Barker for valuable suggestions. The author also thanks the reviewer for a careful reading of the manuscript.

\section{Calder{\'o}n's method}

In this section, we present properties of the following notion of solution:
\begin{defn}[Calder{\'o}n solution]\label{calderonsol}
	Let $3 < p < \infty$ and $u_0 \in B^{s_p}_{p,p}(\R^3)$ be a divergence-free vector field. Suppose $T > 0$ is finite. We say that a distribution $u$ on $Q_T$ is a \emph{Calder{\'o}n solution} on $Q_T$ with initial data $u_0$ if the following requirements are met:
\begin{equation}
u_0 = U_0 + V_0, \quad u = U + V,
\end{equation}
where
\begin{equation}
	U_0 \in L^2(\R^3), \quad V_0 \in \dot B^{s_q+\varepsilon}_{q,q}(\R^3),
	\label{}
\end{equation}
\begin{equation}
	U \in L^\infty_t L^2_x \cap L^2_t \dot H^1_x(Q_T), \quad V \in \cK^{s_q+\varepsilon}_q(Q_T),
	\label{}
\end{equation}
\begin{equation}
	q > p, \quad 0 < \varepsilon < -s_q,
	\label{}
\end{equation}
and $V$ is the mild solution of the Navier-Stokes equations on $Q_T$ with initial data $V_0$, see Theorem \ref{subcrittheory}.
In addition, $U$ satisfies the perturbed Navier-Stokes system
	\begin{equation}\label{nselot}
	\p_t U - \Delta U + \div U \otimes U + \div U \otimes V + \div V \otimes U = - \nabla P, \quad 
	\div U = 0
\end{equation}
in the sense of distributions on $Q_T$, where
\begin{equation}
	P \in L^2_t L^{3/2}_x(Q_T) + L^2(Q_T).
	\label{}
\end{equation}
We require that $U(\cdot,t)$ is weakly continuous as an $L^2(\R^3)$-valued function on $[0,T]$ and that
$U$ attains its initial condition strongly in $L^2(\R^3)$:
\begin{equation}
	\lim_{t \dto 0} \norm{U(\cdot,t)-U_0}_{L^2(\R^3)} = 0.
\end{equation}
Define 
\begin{equation}
	Q := (-\Delta)^{-1} \div \div V \otimes V, \quad p := P + Q.
	\label{}
\end{equation}
We require that $(u,p)$ is suitable for the Navier-Stokes equations:
	\begin{equation}\label{localenergyineq}
		\p_t |u|^2 + 2|\nabla u|^2 \leq \Delta |u|^2 - \div ((|u|^2+2p)u),
\end{equation}
and that $(U,P)$ is suitable for the equation \eqref{nselot}:
	\begin{equation}\label{suitability}
		\p_t |U|^2 + 2|\nabla U|^2 \leq $$ $$ \leq \Delta |U|^2 - \div ((|U|^2 + 2P)U)
		- \div (|U|^2 V) - 2 U \div (V \otimes U).
\end{equation}
The inequalities \eqref{localenergyineq} and \eqref{suitability} are interpreted in the sense of distributions evaluated on non-negative test functions $0 \leq \varphi \in C^\infty_0(Q_T)$.
Lastly, we require that $U$ satisfies the global energy inequality,
	\begin{equation}\label{globalenergyineq}
	\int_{\R^3} |U(x,t_1)|^2 \, dx + 2 \int_{t_0}^{t_1} \int_{\R^3} |\nabla U(x,t)|^2 \, dx \, dt \leq $$ $$ \leq \int_{\R^3} |U(x,t_0)|^2 \, dx + 2 \int_{t_0}^{t_1} \int_{\R^3} V \otimes U : \nabla U dx \, dt,
	\end{equation}
for almost every $0 \leq t_0 < T$, including $t_0 = 0$, and for all $t_1 \in (t_0,T]$.
\end{defn}


\subsection{Splitting arguments}

The next lemma allows us to represent critical initial data as the sum of subcritical and supercritical initial data while preserving the divergence free condition. See Proposition 2.8 in \cite{barkernew} for a detailed proof


\begin{lem}[Splitting of critical data]\label{splitdata}
	 Let $3 < p < q \leq \infty$ and $\theta \in (0,1)$ satisfying
 \begin{equation}
 \frac{1}{p} = \frac{\theta}{2} + \frac{1-\theta}{q}.
 \end{equation}
Define $s := s_p/(1-\theta)$. For all $\lambda > 0$ and divergence-free vector fields $u \in \dot B^{s_p}_{p,p}(\R^3)$,
	there exist divergence-free vector fields $U$, $V$ such that $u = U + V$,
	\begin{equation}\label{splitresult}
	\norm{U}_{L^2(\R^3)} \leq c\norm{u}_{\dot B^{s_p}_{p,p}(\R^3)}^{p/2} \lambda ^{1-p/2},
\end{equation}
\begin{equation}
	\norm{V}_{\dot B^{s}_{q,q}(\R^3)} \leq c \norm{u}_{\dot B^{s_p}_{p,p}(\R^3)}^{p/q} \lambda ^{1-p/q},
\end{equation}
where $c>0$ is an absolute constant.
\end{lem}

The proof is by decomposing the Littlewood-Paley projections as
\begin{equation}
	\dot \Delta_j u = (\dot \Delta_j u) \inds{|\dot \Delta_j u| > \lambda_j} + (\dot \Delta_j u) \inds{|\dot \Delta_j u| \leq \lambda_j}, \quad j \in \Z,
	\label{}
\end{equation}
with an appropriate choice of $\lambda_j > 0$, $j \in \Z$.
The divergence-free condition is kept by applying the Leray projector to the resulting vector fields. Recall that the Leray projector is a Fourier multiplier with matrix-valued symbol homogeneous of degree zero and smooth away from the origin:
\begin{equation}
	(\bP)_{ij} := \delta_{ij} + R_i R_j, \quad R_i := \frac{\p_i}{|\nabla|}, \quad 1 \leq i,j \leq 3.
	\label{}
\end{equation}

Note that $\dot B^{s}_{q,q}(\R^3)$ is indeed a subcritical space of initial data, since
\begin{equation}
s-\frac{3}{q} = -1+ \frac{\theta}{2(1-\theta)} > -1.
\end{equation}
We will often denote $\varepsilon := s-s_q > 0$.

\subsection{Existence of energy solutions to NSE with lower order terms}

In this section, we will prove the existence of weak solutions to the Navier-Stokes equations with coefficients in critical Lebesgue spaces. The method of proof is well known and goes back to Leray \cite{leray}.

\begin{prop}[Existence of energy solutions]\label{energysol}
	Let $U_0 \in L^2(\R^3)$ be a divergence-free vector field, and let
	\begin{equation}
		a,b \in L^l_t L^r_x(Q_T), \quad \frac{2}{l} + \frac{3}{r} = 1, \quad 3 < r \leq \infty
		\label{}
	\end{equation}
	be vector fields for a given $T>0$. Further assume that $\div b = 0$. Then there exist a vector field $U$ and pressure $P$,
	\begin{equation}
		U \in L^\infty_t L^2_x \cap L^2_t \dot H^1_x(Q_T), \quad P \in L^2_t L^{3/2}_x(Q_T) + L^2(Q_T),
		\label{}
	\end{equation}
	such that the perturbed Navier-Stokes system
	\begin{equation}\label{nseloworder}
	\p_t U - \Delta U + \div U \otimes U + \div U \otimes b + \div a \otimes U = - \nabla P, \quad \div U = 0
\end{equation}
is satisfied on $Q_T$ in the sense of distributions. In addition, $U(\cdot,t)$ is weakly continuous as an $L^2(\R^3)$-valued function on $[0,T]$, and
\begin{equation}
	\lim_{t \dto 0} \norm{U(\cdot,t)-U_0}_{L^2(\R^3)} = 0.
	\label{}
\end{equation}
Finally, $(U,P)$ is suitable for the equation \eqref{nseloworder}:
	\begin{equation}\label{localenergy}
	\p_t |U|^2 + 2 |\nabla U|^2 \leq $$ $$ \leq \Delta |U|^2 - \div((|U|^2 + 2P)U)
	- \div (|U|^2b) - 2 U \div (a \otimes U)
\end{equation}
as distributions evaluated on non-negative test functions $0 \leq \varphi \in C^\infty_0(Q_T)$, and
$U$ satisfies the global energy inequality
\begin{equation}
\label{globalenergy}
	\int_{\R^3} |U(x,t_2)|^2 \, dx \, dt + 2 \int_{t_1}^{t_2} \int_{\R^3} |\nabla U(x,t)|^2 \, dx \, dt \leq
	 $$ $$ \leq \int_{\R^3} |U(x,t_1)|^2 \, dx + 2\int_{t_1}^{t_2} \int_{\R^3} a \otimes U : \nabla U \, dx \, dt
\end{equation}
for almost every $t_1 \in [0,T)$, including $t_1 = 0$, and for all $t_2 \in (t_1,T]$.
\end{prop}
In the statement above, $U \div (a \otimes U)$ is the distribution
\begin{equation}
	\la U \div (a \otimes U), \varphi \ra := - \int_0^T \int_{\R^3} a \otimes U : (U \otimes \nabla \varphi + \varphi \nabla U) \, dx \, dt,
\end{equation}
for all $\varphi \in C^\infty_0(Q_T)$.

Let us introduce some notation and basic estimates surrounding the energy space. For $0 \leq t_0 < T \leq \infty$, we define $Q_{t_0,T} := \R^3 \times (t_0,T)$.
If $U \in L^\infty_t L^2_x \cap L^2_t \dot H^1_x(Q_{t_0,T})$, we define the \emph{energy norm},
\begin{equation}
	|U|_{2,Q_{t_0,T}}^2 := \sup_{t_0<t<T} \int_{\R^3} |U(x,t)|^2 \, dx + 2 \int_{t_0}^T \int_{\R^3} |\nabla U(x,t')|^2 \,dx \, dt'.
\end{equation}
For simplicity, take $t_0 = 0$.
By interpolation between $L^\infty_t L^2_x(Q_{T})$ and $L^2_t L^6_x(Q_{T})$, one obtains
\begin{equation}
	\norm{U}_{L^m_t L^n_x(Q_{T})} \leq c |U|_{2,Q_{T}},
	\label{}
\end{equation}
\begin{equation}
	\frac{2}{m} + \frac{3}{n} = \frac{3}{2}, \quad 2 \leq m \leq \infty, \quad 2 \leq n \leq 6.
	\label{}
\end{equation}
For instance, it is common to take $m=n=10/3$.
Hence, if $a \in L^l_t L^r_x(Q_{T})$ is a vector field as in Proposition \ref{energysol}, we have
\begin{equation}\label{energyinterp1}
	\int_{0}^T |U|^2 |a|^2 \, dx \, dt \leq c \norm{a}_{L^l_t L^r_x(Q_{T})}^2 |U|^2_{2,Q_{T}},
\end{equation}
\begin{equation}\label{energyinterp2}
	\int_{0}^T |U|^2 |a| \, dx \,dt \leq c T^{\frac{1}{2}} \norm{a}_{L^l_t L^r_x(Q_{T})} |U|^2_{2,Q_{T}}.
\end{equation}

Let us now recall the tools used to estimate the time derivative $\p_t U$ of a solution $U$ that belongs to the energy space. To start, we will need to estimate the pointwise product of $u \in L^2(\R^3)$ and $v \in \dot H^1(\R^3)$:
\begin{equation}\label{energyinterp3}
	\norm{uv}_{\dot H^{-1/2}(\R^3)} \leq c \norm{u}_{L^2(\R^3)} \norm{v}_{\dot H^1(\R^3)},
\end{equation}
see Corollary 2.55 in \cite{bahourichemin}. For instance, it follows that
\begin{equation}
	\norm{\div U \otimes U}_{L^2_t \dot H^{-3/2}_x(Q_T)} \leq c |U|_{2,Q_T}^2.
	\label{}
\end{equation}
The time derivative $\p_t U$ is typically only in $L^2_t \dot H^{-3/2}_x(Q_T)$ unless the nonlinear term is mollified. Notice also that $\dot H^{-3/2}(\R^3) \into H^{-3/2}(\R^3)$.
Suppose $(U^{(n)})_{n \in \N}$ is a sequence of vector fields on $Q_T$ such that
\begin{equation}
	\sup_{n \in \N} |U^{(n)}|_{2,Q_T} + \norm{\p_t U}_{L^2_t H^{-3/2}_x(Q_T)} < \infty.
	\label{}
\end{equation}
According to the Aubin-Lions lemma (see Chapter 5, Proposition 1.1 in \cite{sereginnotes}), there exists a subsequence, still denoted by $U^{(n)}$, such that
\begin{equation}
	U^{(n)} \to U \text{ in } L^2(B(R) \times (0,T)), \quad R > 0.
	\label{}
\end{equation}
Since $\sup_{n \in \N} \norm{U^{(n)}}_{L^{10/3}(Q_T)} < \infty$, the subsequence actually converges strongly in $L^3(B(R) \times (0,T))$.
We will use this fact frequently in the sequel.

For the remainder of the paper, let us fix a non-negative radially symmetric test function $0 \leq \theta \in C^\infty_0(\R^3)$ with $\int_{\R^3} \theta \, dx = 1$. For a locally integrable function $f \in L^1_\loc(\R^3)$ we make the following definition:
\begin{equation}
	\theta_\rho := \frac{1}{\rho^{3}} \theta\big(\frac{\cdot}{\rho}\big), \quad (f)_\rho := f \ast \theta_\rho, \quad \rho > 0.
	\label{}
\end{equation}
The proof of the next lemma is well known, and we include it for completeness.
\begin{lem}[Solution to mollified Navier-Stokes equation with lower order terms]\label{stokes}
	Let $\rho > 0$. Assume the hypotheses of Proposition \ref{energysol}. Then there exists a unique $U$ in the class
	\begin{equation}
		U \in C([0,T];L^2(\R^3)) \cap L^2_t \dot H^1_x(Q_T), \quad \p_t U \in L^2_t H^{-1}_x(Q_T),
	\end{equation}
satisfying the mollified perturbed Navier-Stokes system in $L^2_t H^{-1}_x(Q_T)$:
\begin{equation}\label{eq:lemperturbedsys}
		\p_t U - \Delta U + \bP \div (U \otimes (U)_\rho) + \bP \div(U \otimes b + a \otimes U) = 0, \quad \div U = 0,
\end{equation}
and such that $U(\cdot,0) = U_0$ in $L^2(\R^3)$.
\end{lem}
\begin{proof}
	Assume the hypotheses of the lemma and denote $T_\sharp := T$ in the statement, in order to reuse the variable $T$. We will consider the bilinear operator
	\begin{equation}
		B_\rho(v,w)(\cdot,t) := \int_0^t e^{(t-s)\Delta} \bP \div v \otimes (w)_\rho \, ds,
	\end{equation}
	defined formally for all vector fields $v$, $w$ on spacetime, as well as the linear operator
	\begin{equation}
		L(w)(\cdot,t) := \int_0^t e^{(t-s)\Delta} \bP \div (a \otimes w + w \otimes b) \,ds.
		\label{}
	\end{equation}
	For instance, by classical energy estimates for the Stokes equations, the operators are well defined whenever $v \otimes (w)_\rho$, $a \otimes w$, $w \otimes b$ are square integrable. Specifically, due to \eqref{energyinterp1} and the energy estimates, we have that $B$ and $L$ are bounded operators on the energy space, and
	\begin{equation}
		|B(v,w)|_{2,Q_T}^2 \leq c \norm{v \otimes (w)_\rho}_{L^2(Q_T)}^2 \leq c(\rho) T |v|_{2,Q_T}^2 |w|_{2,Q_T}^2, 
		\label{}
	\end{equation}
	\begin{equation}
		|L(w)|_{2,Q_T}^2 \leq c \norm{a \otimes w}_{L^2(Q_T)}^2 \leq \norm{a}_{L^l_t L^r_x(Q_T)}^2 |w|^2_{2,Q_T},
		\label{}
	\end{equation}
	for all $T>0$. Notice that $b$ does not enter the estimate, since $\div b = 0$. In addition, the operators take values in $C([0,T];L^2(\R^3)$. Moreover, notice that $\norm{a}_{L^l_t L^r_x(Q_T)} \ll 1$ when $0 < T \ll 1$. Hence, one may apply Lemma \ref{abstract} to solve the integral equation
	\begin{equation}\label{integraleqnenergy}
		U(\cdot,t) = e^{t\Delta} U_0 - B_\rho(U,U)(\cdot,t) - L(U)(\cdot,t)
	\end{equation}
	up to time $0 < T \ll 1$. The integral equation \eqref{integraleqnenergy} is equivalent to the differential equation \eqref{eq:lemperturbedsys}. The solution may be continued in the energy space up to the time $T_\sharp$ by the same method as long as the energy norm remains bounded. This will be the case, since a solution $U$ on $Q_T$ obeys the energy equality
	\begin{equation}
		\int_{\R^3} |U(x,t_2)|^2 \, dx + 2 \int_{t_1}^{t_2} \int_{\R^3} |\nabla U|^2 \, dx \, dt = $$ $$ = \int_{\R^3} |U(x,t_1)|^2 \, dx + 2 \int_{t_1}^{t_2} \int_{\R^3} a \otimes U : \nabla U \, dx \, dt \leq $$ $$
		\leq \int_{\R^3} |U(x,t_1)|^2 \, dx + A(t_1,t_2) |U|^2_{2,Q_{t_1,t_2}}
	\end{equation}
	for all $0 \leq t_1 < t_2 \leq T$, where $A(t_1,t_2) := c \norm{a}_{L^l_t L^r_x(Q_{t_1,t_2})}$. Then one simply takes $t_2$ close enough to $t_1$ such that $A(t_1,t_2) \leq 1/2$ to obtain a local-in-time \emph{a priori} energy bound. By repeating the argument a finite number of times, one obtains
	\begin{equation}
		|U|_{2,Q_T} \leq C(\norm{U_0}_{L^2(\R^3)},\norm{a}_{L^l_t L^r_x(Q_{T_\sharp})}).
		\label{}
	\end{equation}
	Hence, the solution may be continued to the time $T_\sharp$.
	Finally, uniqueness follows from the construction. This completes the proof.
\end{proof}

\begin{proof}[Proof of Proposition \ref{energysol}]
	We follow the standard procedure initiated by Leray \cite{leray} of solving the mollified problem and taking the limit as $\rho \dto 0$. The arguments we present here are essentially adapted from \cite{sverakl3} and \cite{sereginnotes}, Chapter 5, so we will merely summarize them. 

	1. \textit{Limits}. 
	For each $\rho > 0$, we denote by $U^\rho$ the unique solution from Lemma \ref{stokes}. From the proof of Lemma \ref{stokes}, recall the energy equality
	\begin{equation}\label{energyeq}
		\int_{\R^3} |U^\rho(x,t_2)|^2 \, dx + 2 \int_{t_1}^{t_2} \int_{\R^3} |\nabla U^\rho|^2 \, dx \, dt = $$ $$
		= \int_{\R^3} |U^\rho(x,t_1)|^2 \, dx + 2 \int_{t_1}^{t_2}  \int_{\R^3} a \otimes U^\rho : \nabla U^\rho \, dx \, dt
\end{equation}
for all $0 \leq t_1 < t_2 \leq T$, as well as the \emph{a priori} energy bound
\begin{equation}\label{eq:apriorienergy}
|U^\rho|_{2,Q_T} \leq C(\norm{U_0}_{L^2(\R^3)},\norm{a}_{L^l_t L^r_x(Q_T)}).
\end{equation}
In order to estimate the time derivative $\p_t U \in L^2_t H^{-3/2}_x(Q_T)$, we rewrite the equation \eqref{eq:lemperturbedsys} as
\begin{equation}
	\p_t U = \Delta U - \bP \div (U^\rho \otimes (U^\rho)_\rho) - \bP \div (a \otimes U^\rho + U^\rho \otimes b).
	\label{}
\end{equation}
Then, due to the estimate \eqref{energyinterp1}, one obtains
\begin{equation}
	\norm{\Delta U - \bP \div (a \otimes U^\rho + U^\rho \otimes b)}_{L^2_t H^{-1}_x(Q_T)} \leq $$ $$\leq C(\norm{U_0}_{L^2(\R^3)},\norm{a}_{L^l_t L^r_x(Q_T)},\norm{b}_{L^l_t L^r_x(Q_T)}),
\end{equation}
and in light of \eqref{energyinterp3}, we also have
\begin{equation}
	\norm{\bP \div (U^\rho \otimes (U^\rho)_\rho)}_{L^2_t H^{-3/2}_x(Q_T)} \leq c \norm{(U^\rho \otimes (U^\rho)_\rho)}_{L^2_t H^{-1/2}_x(Q_T)} \leq $$ $$ \leq C(\norm{U_0}_{L^2(\R^3)},\norm{a}_{L^l_t L^r_x(Q_T)}).
\end{equation}
The resulting estimate on the time derivative is
\begin{equation}\label{eq:aprioritimederiv}
	\norm{\p_t U^\rho}_{L^2_t H^{-3/2}_x(Q_T)} \leq C(\norm{U_0}_{L^2(\R^3)},\norm{a}_{L^l_t L^r_x(Q_T)},\norm{b}_{L^l_t L^r_x(Q_T)}).
\end{equation}
By the Banach-Alaoglu theorem and \eqref{eq:apriorienergy}, there exist $U \in L^\infty_t L^2_x \cap L^2_t \dot H^1_x(Q_T)$ and a sequence $\rho_k \dto 0$ such that
	\begin{equation}
		U^{\rho_k} \wstar U \quad \text{ in } L^\infty_t L^2_x(Q_T), \qquad
		\nabla U^{\rho_k} \wto \nabla U \quad \text{ in } L^2(Q_T).
\end{equation}
In addition, by \eqref{eq:aprioritimederiv} and the discussion preceding Lemma \ref{stokes}, the subsequence may be chosen such that
\begin{equation}
U^{\rho_k} \to U \quad \text{ in } L^3_t (L^3_\loc)_x(Q_T).
\end{equation}
Moreover, the estimate \eqref{eq:aprioritimederiv} allows us to prove that
\begin{equation}\label{Ufnalconv}
	\int_{\R^3} U^{\rho_k}(x,\cdot) \cdot \varphi(x) \, dx  \to  \int_{\R^3} U(x,\cdot) \cdot \varphi(x) \, dx  \text{ in } C([0,T]),
\end{equation}
 for all vector fields $\varphi \in L^2(\R^3)$. In particular, 
	\begin{equation}
		U^{\rho_k}(\cdot,t) \wto U(\cdot,t) \text{ in } L^2(\R^3), \quad t \in [0,T],
	\end{equation}
	and the limit $U(\cdot,t)$ is weakly continuous as an $L^2(\R^3)$-valued function on $[0,T]$. Here is our argument for \eqref{Ufnalconv}, based on Chapter 5, p. 102 of \cite{sereginnotes}. For each vector field $\varphi \in C^\infty_0(\R^3)$, consider the family
\begin{equation}
	\cF_\varphi := \big\lbrace \int_{\R^3} U^\rho(x,\cdot) \cdot \varphi(x) \, dx : \rho > 0 \big\rbrace \subset C([0,T]).
	\label{}
\end{equation}
The family $\cF_\varphi$ is uniformly bounded, since
	\begin{equation}
		\lvert \int U^\rho(x,t) \cdot \varphi(x) \, dx \rvert \leq \norm{U^\rho}_{L^\infty_t L^2_x(Q_T)} \norm{\varphi}_{L^2(\R^3)} \leq $$ $$
		\leq C(\norm{U_0}_{L^2(\R^3)},\norm{a}_{L^l_t L^r_x(Q_T)}) \norm{\varphi}_{L^2(\R^3)},
	\end{equation}
	for all $0 \leq t \leq T$. It is also equicontinuous:
	\begin{equation}
		\lvert \int U^\rho(x,t_2) \cdot \varphi(x) - U^\rho(x,t_1) \cdot \varphi(x) \, dx \rvert \leq $$ $$
	\leq |t_2-t_1|^{\frac{1}{2}} \norm{\p_t U^\rho}_{L^2_t H^{-3/2}_x(Q_T)} \norm{\varphi}_{H^{3/2}(\R^3)} \leq $$ $$
	\leq |t_2-t_1|^{\frac{1}{2}} C(\norm{U_0}_{L^2(\R^3)},\norm{a}_{L^l_t L^r_x(Q_T)},\norm{b}_{L^l_t L^r_x(Q_T)}) \norm{\varphi}_{H^{3/2}(\R^3)},
	\end{equation}
	for all $0 \leq t_1,t_2 \leq T$. Hence, we may extract a further subsequence, still denoted by $\rho_k$, such that \eqref{Ufnalconv}. By a diagonalization argument and the density of test functions in $L^2(\R^3)$, we can obtain \eqref{Ufnalconv} for all vector fields $\varphi \in L^2(\R^3)$. This completes the summary of the convergence properties of $U^{\rho_k}$ as $\rho_k \dto 0$.

Let us now analyze the behavior of $U$ near the initial time. In the limit $\rho_k \dto 0$, the energy equality \eqref{energyeq} gives rise to an energy \emph{inequality}:
\begin{equation}\label{eq:energyineq}
		\int_{\R^3} |U(x,t)|^2 \, dx + 2 \int_0^t \int_{\R^3} |\nabla U|^2 \, dx \, dt' \leq $$ $$
		\leq \int_{\R^3} |U_0(x)|^2 \, dx + C(\norm{U_0}_{L^2(\R^3)},\norm{a}_{L^l_t L^r_x(Q_T)}) \norm{a}_{L^l_t L^r_x(Q_t)},
\end{equation}
for almost every $t \in (0,T)$.
This is due to the lower semicontinuity of the energy norm with respect to weak-star convergence.
	 Since $U(\cdot,t)$ is weakly continuous as an $L^2(\R^3)$-valued function, the energy inequality may be extended to all $t \in (0,T]$. Finally, since $U(\cdot,0) = U_0$ and $\limsup_{t \dto 0} \norm{U(\cdot,t)}_{L^2(\R^3)} \leq \norm{U_0}_{L^2(\R^3)}$ (by taking $t \dto 0$ in \eqref{eq:energyineq}), we obtain
	\begin{equation}
		\lim_{t \dto 0} \norm{U(\cdot,t) - U_0}_{L^2(\R^3)} = 0.
		\label{strongconvergencenearzero}
	\end{equation}
	We have proven the desired properties of $U$ except for suitability and the global energy inequality.

	Let us now take the limit of the pressures. The pressure $P^\rho$ associated to $U^\rho$ in the equation \eqref{eq:lemperturbedsys} may be calculated as $P^\rho := P_1^\rho + P_2^\rho$,
	\begin{equation}\label{p1def}
		P_1^\rho := (-\Delta)^{-1} \div \div (U^\rho \otimes (U^\rho)_\rho)
	\end{equation}
	\begin{equation}\label{p2def}
		P_2^\rho := (-\Delta)^{-1} \div \div (U^\rho \otimes b + a \otimes U^\rho).
\end{equation}
By the Calder{\'o}n-Zygmund estimates, we have the following bounds independent of the parameter $\rho > 0$:
	\begin{equation}
		\norm{P_1^\rho}_{L^2_t L^{3/2}_x(Q_T)} \leq c \norm{U^\rho \otimes (U^\rho)_\rho}_{L^2_t L^{3/2}_x(Q_T)} \leq $$ $$
		\leq c\norm{U^\rho}_{L^\infty_t L^2_x(Q_T)} \norm{U^\rho}_{L^2_t L^6_x(Q_T)} \leq $$ $$ 
		\leq C(\norm{U_0}_{L^2(\R^3)},\norm{a}_{L^l_t L^r_x(Q_T)}),
\end{equation}
and similarly,
	\begin{equation}
		\norm{P_2^\rho}_{L^2(Q_T)} \leq c \norm{a \otimes U^\rho + U^\rho \otimes b}_{L^2(Q_T)} \leq $$ $$
		\leq c |U^\rho|_{2,Q_T} (\norm{a}_{L^l_t L^r_x(Q_T)} + \norm{b}_{L^l_t L^r_x(Q_T)}) \leq $$ $$
		\leq C(\norm{U_0}_{L^2(\R^3)},\norm{a}_{L^l_t L^r_x(Q_T)},\norm{b}_{L^l_t L^r_x(Q_T)}).
\end{equation}
	In particular, we may pass to a further subsequence, still denoted by $\rho_k$, such that
	\begin{equation}
		P_1^{\rho_k} \wto P_1 \quad \text{ in } L^2_t L^{3/2}_x(Q_T), \quad
		P_2^{\rho_k} \wto P_2 \quad \text{ in } L^2(Q_T),
\end{equation}
	\begin{equation}
		P^{\rho_k} \wto P := P_1 + P_2 \quad \text{ in } L^2_t (L^{3/2}_\loc)_x(Q_T).
	\end{equation}
	Recall that $U^{\rho_k} \to U$ in $L^{3}_t (L^3_\loc)_x$. Utilizing this fact in \eqref{p1def} and \eqref{p2def}, we observe that
	\begin{equation}
		-\Delta P_1 = \div \div U \otimes U, \quad -\Delta P_2 = \div \div (a \otimes U + U \otimes b),
		\label{}
	\end{equation}
	in the sense of distributions on $Q_T$. Finally, the Liouville theorem for entire harmonic functions in Lebesgue spaces implies
	\begin{equation}
		P_1 = (-\Delta)^{-1} \div \div U \otimes U, \quad P_2 = (-\Delta)^{-1} \div \div (a \otimes U + U \otimes b).
	\label{}
\end{equation}
Due to the limit behavior discussed above, it is clear that $(U,P)$ solves the system \eqref{nseloworder} in the sense of distributions.

	2. \textit{Suitability}.
	We will now prove that $U$ is suitable for the equation \eqref{nseloworder}. Specifically, we will verify the local energy inequality \eqref{localenergy} following arguments in Lemari{\'e}-Rieusset \cite{lemarie1} (see p. 318).
	We start from the mollified local energy equality
	\begin{equation}\label{suitableeq}
		\p_t |U^{\rho}|^2 + 2|\nabla U^{\rho}|^2 = $$ $$= \Delta |U^{\rho}|^2 - \div (|U^{\rho}|^2(U^{\rho})_{\rho} + 2 P^{\rho} U^{\rho})
		- \div (|U^{\rho}|^2 b) - 2 U^{\rho} \div (a \otimes U^{\rho}),
\end{equation}
which is satisfied in the sense of distributions by solutions of \eqref{eq:lemperturbedsys} on $Q_T$.
Let us analyze the convergence of each term in \eqref{suitableeq} as $\rho_k \dto 0$. Recall that 
\begin{equation}
	P^{\rho_k} \wto P \text{ in } L^{3/2}_\loc(Q_T), \quad U^{\rho_k} \to U \text{ in } L^3_\loc(Q_T).
	\label{}
\end{equation}
This readily implies
\begin{equation}
	|U^{\rho_k}|^2 (U^{\rho_k})_{\rho_k} \to |U|^2 U \text{ in }L^1_\loc(Q_T), \quad
	P^{\rho_k} U^{\rho_k} \to PU \text{ in } L^1_\loc(Q_T),
\end{equation}
\begin{equation}
	|U^{\rho_k}|^2 \to |U|^2 \text{ in } L^{3/2}_\loc(Q_T).
	\label{}
\end{equation}
Moreover, according to the estimates \eqref{energyinterp1} and \eqref{energyinterp2}, \begin{equation}
	|U^{\rho_k}|^2 b \to |U|^2 b \text{ in } L^1(Q_T),
\end{equation}
\begin{equation}
	U^{\rho_k} \div (a \otimes U^{\rho_k}) \to U \div (a \otimes U) \text{ in } \cS'(Q_T).
	\label{}
\end{equation}
It remains to analyze the term $|\nabla U^{\rho_k}|^2$. Upon passing to a subsequence, $|\nabla U^{\rho_k}|^2$ converges weakly-star in $\cM(Q_T)$, the space consisting of finite Radon measures on $Q_T$, but its limit may not be $|\nabla U|^2$. On the other hand, recall that $\nabla U^{\rho_k} \wto \nabla U$ in $L^2(Q_T)$. Hence, by lower semicontinuity of the $L^2$ norm with respect to weak convergence,
\begin{equation}
	\liminf_{\rho_k \dto 0} \int_E |\nabla U^{\rho_k}|^2 - |\nabla U|^2 \, dx \geq 0
\end{equation}
for all Borel measurable sets $E$ contained in $Q_T$. Therefore, upon passing to a further subsequence,
\begin{equation}
	\mu := \lim_{\rho_k \dto 0} |\nabla U^{\rho_k}|^2 - |\nabla U|^2
\end{equation}
is a non-negative finite measure on $Q_T$, and \eqref{suitableeq} becomes
\begin{equation}
	\p_t |U|^2 + 2|\nabla U|^2 = $$ $$ = \Delta |U|^2 - \div ((|U|^2 + 2P)U)	- \div (|U|^2 b) - 2 U \div (a \otimes U) - 2\mu.
\end{equation}

3. \textit{Global energy inequality}. Finally, let us pass from the local energy inequality \eqref{localenergy} to the global energy inequality \eqref{globalenergy} in the following standard way (see Lemari{\'e}-Rieusset \cite{lemarie1}, p. 319).

Let $0 \leq \eta \in C^\infty_0(\R)$ be an even function such that $\eta \equiv 1$ for $|t| \leq 1/4$, $\eta \equiv 0$ for $|t| \geq 1/2$, and $\int_\R \eta \, dt = 1$. Define $\eta_\varepsilon(t) := \varepsilon^{-1} \eta(\varepsilon^{-1} t)$. Given $0 < t_0 < t_1 < T$, consider
\begin{equation}
	\psi_\varepsilon(t) := \int_{-\infty}^t \eta_\varepsilon(t'-t_0) - \eta_\varepsilon(t'-t_1) \, dt', \quad t \in \R.
\end{equation} The functions $\psi_\varepsilon$ are smooth approximations of the characteristic function $\ind_{(t_1,t_2)}$. Now let $0 \leq \varphi \in C^\infty_0(\R^3)$ such that $\varphi \equiv 1$ on $B(1)$ and $\varphi \equiv 0$ outside $B(2)$. Define 
\begin{equation}
	\Phi_{\varepsilon,R}(x,t) := \psi_\varepsilon(t) \varphi\big(\frac{x}{R}\big)^2, \quad (x,t) \in \R^{3+1}.
\end{equation}
Using $\Phi_{\varepsilon,R}$ in the local energy inequality \eqref{localenergy} with $0 < \varepsilon \ll 1$, we have
	\begin{equation}
		-\int \frac{d\psi_\varepsilon}{dt} \varphi^2\big(\frac{x}{R}\big) |U|^2 + 2\int \psi_\varepsilon \varphi^2\big(\frac{x}{R}\big) |\nabla U|^2 \leq $$ $$
		\leq \int \psi_\varepsilon \Delta \Big( \varphi^2\big(\frac{x}{R}\big) \Big) |U|^2 + \frac{2}{R} \int \psi_\varepsilon \varphi\big(\frac{x}{R}\big) (\nabla\varphi)\big(\frac{x}{R}\big) \cdot |U|^2 U $$ $$
		+ \frac{4}{R} \int \psi_\varepsilon \varphi\big(\frac{x}{R}\big) (\nabla \varphi)\big(\frac{x}{R}\big) \cdot PU + \frac{2}{R} \int \psi_\varepsilon \varphi\big(\frac{x}{R}\big) (\nabla \varphi)\big(\frac{x}{R}\big) \cdot |U|^2 b $$ $$
		+ \frac{4}{R} \int \psi_\varepsilon \varphi\big(\frac{x}{R}\big) \, a \otimes U : U \otimes (\nabla \varphi)\big(\frac{x}{R}\big) + 2 \int \psi_\varepsilon \varphi^2\big(\frac{x}{R}\big) \, a \otimes U : \nabla U,
	\end{equation}
	where all the integrals are taken over $Q_T$. Since $U$ is in the energy space and $P \in L^2_t L^{3/2}_x(Q_T) + L^2(Q_T)$, we may take the limit as $R \upto \infty$ to obtain
	\begin{equation}
		-\int \frac{d\psi_\varepsilon}{dt} |U|^2 + 2\int \psi_\varepsilon |\nabla U|^2 \leq 2 \int \psi_\varepsilon a \otimes U : \nabla U.
	\end{equation}
	Moreover, if $t_0, t_1$ are Lebesgue points of $\norm{U(\cdot,t)}_{L^2(\R^3)}$, then in the limit as $\varepsilon \dto 0$,
	\begin{equation}
		\label{globalenergyineqawayfrom}
		\int_{\R^3} |U(t_1)|^2 + 2 \int_{t_0}^{t_1} \int_{\R^3} |\nabla U|^2 \leq \int_{\R^3} |U(t_0)|^2 +2 \int_{t_0}^{t_1} \int_{\R^3} a \otimes U : \nabla U.
	\end{equation}
	The case when the initial time is zero is recovered from \eqref{globalenergyineqawayfrom} by taking the limit as $t_0 \dto 0$, since $\lim_{t \dto 0} \norm{U(\cdot,t) - U_0}_{L^2(\R^3)} = 0$ was already demonstrated in \eqref{strongconvergencenearzero}. The proof is complete.
\end{proof}

\subsection{Existence of Calder{\'o}n solutions}

We are now ready to prove the existence of Calder{\'o}n solutions.

\begin{thm}[Existence of Calder{\'o}n solutions]\label{existence}
	Let $T > 0$ and $3 < p < \infty$. Suppose $u_0 \in \dot B^{s_p}_{p,p}(\R^3)$ is a divergence-free vector field. Then there exists a Calder{\'o}n solution $u$ on $Q_T$ with initial data $u_0$.
\end{thm}
\begin{proof}
	1. \textit{Splitting arguments}.
		Assume the hypotheses of the theorem, and let $q \in (p,\infty)$. According to Lemma \ref{splitdata}, there exists $0 < \varepsilon <-s_q$ such that for all $M>0$, we may decompose the initial data as follows:
	\begin{equation}\label{eq:decomposeme}
		u_0 = U_0 + V_0,
	\end{equation}
where
	\begin{equation}
		U_0 \in L^2(\R^3), \quad \norm{U_0}_{L^2(\R^3)} \leq C(\norm{u_0}_{\dot B^{s_p}_{p,p}(\R^3)},M),
		\label{}
	\end{equation}
	\begin{equation}
		V_0 \in \dot B^{s_q+\varepsilon}_{q,q}(\R^3), \quad \norm{V_0}_{\dot B^{s_q+\varepsilon}_{q,q}(\R^3)} < M.
		\label{}
	\end{equation}
	The decomposition depends on $M > 0$.
	By Theorem \ref{subcrittheory}, there exists a constant $\gamma := \gamma(q,\varepsilon,T) > 0$ such that whenever
	\begin{equation}
		\norm{V_0}_{\dot B^{s_q+\varepsilon}_{q,q}(\R^3)} \leq \gamma,
		\label{}
	\end{equation}
	there exists a unique mild solution $V \in \cK^{s_q+\varepsilon}_{q}(Q_{T})$ of the Navier-Stokes equations on $Q_T$ with initial data $V_0$, and the mild solution obeys
	\begin{equation}\label{katoimplies}
	\norm{V}_{\cK^{s_q+\varepsilon}_{q}(Q_{T})} \leq c(q,\varepsilon) \norm{V_0}_{\dot B^{s_q+\varepsilon}_{q,q}(\R^3)}.
\end{equation}
Let $M = \gamma$ when forming the decomposition \eqref{eq:decomposeme}. The corresponding mild solution $V$ with initial data $V_0$ exists on $Q_T$ and satisfies \eqref{katoimplies}. Hence,
\begin{equation}
	V \in L^l_t L^q_x(Q_T), \quad \frac{2}{l} + \frac{3}{q} = 1.
	\label{}
\end{equation}
Let $U \in L^\infty_t L^2_x \cap L^2_t \dot H^1_x(Q_T)$ be a solution constructed in Proposition \ref{energysol} to the perturbed Navier-Stokes equation \eqref{nseloworder} with initial data $U_0 \in L^2(\R^3)$ and coefficients $a=b=V$. If $u=U+V$ satisfies the local energy inequality \eqref{localenergyineq}, then $u$
is a Calder{\'o}n solution on $Q_T$ with initial data $u_0$.

2. \textit{Suitability of full solution}. Let us return to the approximation procedure for $U$ in Proposition \ref{energysol}. Define
\begin{equation}
	u^\rho := U^{\rho} + V, \quad p^\rho := P^{\rho} + Q, \quad Q := (-\Delta)^{-1} \div \div V \otimes V.
	\label{}
\end{equation}
We will prove that the solution $u=U+V$ with pressure $p=P+Q$ satisfies
	\begin{equation}
		\p_t |u|^2 + 2|\nabla u|^2 = \Delta |u|^2 - \div ((|u|^2+2p)u) - 2\mu
	\end{equation}
	in the sense of distributions on $Q_T$, where $\mu$ is a finite non-negative measure on $Q_T$. Let us start from the mollified local energy equality for $u^{\rho}$,
\begin{equation}
	\p_t |u^{\rho}|^2 + 2|\nabla u^{\rho}|^2 = \Delta |u^{\rho}|^2 - \div (|u^{\rho}|^2(u^{\rho})_{\rho} + 2 p^{\rho} u^{\rho}) $$ $$
	- u^{\rho} \, \div \left( u^{\rho} \otimes (V-(V)_\rho) + V \otimes (U^\rho - (U^\rho)_\rho \right).
\end{equation}
This is the same equality as if $u^{\rho}$ solved the mollified Navier-Stokes equations (see, e.g., \cite{lemarie1}, p. 318), except for the second line, which adjusts for the fact that $V$ solves the actual Navier-Stokes equations instead of the mollified equations. The distribution on the second line converges to zero as $\rho_k \dto 0$, and all the other convergence arguments are as in Step 2 of Proposition \ref{energysol}.
\end{proof}

We now demonstrate that there exists a Calder{\'o}n solution which agrees with the mild solution until its maximal time of existence. 

\begin{thm}[Mild Calder{\'o}n solutions]\label{mildcalderon}
	Under the assumptions of Theorem \ref{existence}, there exists a Calder{\'o}n solution $u$ on $Q_T$ with initial data $u_0$ such that $u$ agrees with the mild solution $NS(u_0)$ until the time $\min(T,T^*(u_0))$.
\end{thm}
The mild solution $NS(u_0)$ under consideration is constructed in Theorem \ref{mildexist}.
\begin{proof}
	\textit{1. Introducing the integral equation.}
	Let $U_0,V_0,V$ be as in Step 1 of Theorem \ref{existence}. From now on, we will denote $T_\sharp := T$ in the statement of the theorem, so that the variable $T$ can be reused.

	The set-up is as follows. 
	Recall that for all $0 < T < \min(T_\sharp,T^*)$,
	\begin{equation}
		NS(u_0) \in \mathring{\cK}_p(Q_T) \cap \mathring{\cK}_\infty(Q_T),
		\label{}
	\end{equation}
	\begin{equation}
		V \in \cK_q^{s_q+\varepsilon}(Q_T) \cap \cK_\infty^{-1+\varepsilon}(Q_T).
		\label{}
	\end{equation}
Let us define $\tilde{U} := NS(u_0) - V$.
Then $\tilde{U} \in \mathring{\cK}_q(Q_T) \cap \mathring{\cK}_\infty(Q_T)$ for all $0 < T < \min(T_\sharp,T^*)$ is a mild solution of the integral equation
	\begin{equation}\label{integraleqn}
W(\cdot,t) = e^{t \Delta} U_0 - B(W,W)(\cdot,t) - L(W)(\cdot,t),
	\end{equation}
	where we formally define
	\begin{equation}
		B(v,w)(\cdot,t) := \int_0^t e^{(t-s)\Delta} \bP \div v \otimes w \, ds,
	\end{equation}
	\begin{equation}
		L(w)(\cdot,t) := B(w,V)(\cdot,t) + B(V,w)(\cdot, t),
	\end{equation}
	for all vector fields $v$, $w$ on spacetime. The initial data $U_0$ belongs to the class
	\begin{equation}\label{eq:initialdatabelongs}
		U_0 \in [\dot B^{s_p}_{p,p}(\R^3) + \dot B^{s_q+\varepsilon}_{q,q}(\R^3)] \cap L^2(\R^3).
\end{equation}
First, we will demonstrate that the mild solution of the integral equation \eqref{integraleqn} in $\mathring{\cK_q}(Q_T)$ is unique, where $0 < T < \min(T_\sharp,T^*)$. Notice that in the decomposition $u = U + V$ of a Calder{\'o}n solution, the vector field $U$ formally solves \eqref{integraleqn}. We will show that in the proof of Theorem \ref{existence}, it is possible to construct $U$ in the space $\mathring{\cK}_q(Q_T)$. Therefore, $U$ will satisfy $U \equiv \tilde{U}$ on $Q_T$, and the proof will be complete.


	\textit{2. Existence of mild solutions to integral equation.}
	Let us summarize the local well-posedness of mild solutions $W \in \mathring{\cK}_q(Q_T)$ to the integral equation \eqref{integraleqn}. Our main goal is to establish estimates on the operators $B$ and $L$. Then the local existence of mild solutions in $\mathring{\cK}_q(Q_S)$ for some $0 < {S} \ll 1$ will follow from Lemma~\ref{abstract}. 
	
	Let $0 < T < \min(T_\sharp,T^*)$ and $v,w \in \mathring{\cK}_q(Q_T)$. First, observe that $L \: \mathring{\cK}_q(Q_T) \to \mathring{\cK}_q(Q_T)$ is bounded with operator norm satisfying $\norm{L}_{\mathring{\cK}_q(Q_T)} < 1$ when $0 < T \ll 1$. Indeed, according to the Kato estimates in Lemma \ref{katoest},
	\begin{equation}\label{Lest1}
	\norm{L(v)}_{\cK_q(Q_T)} = \norm{B(v,V) + B(V,v)}_{\cK_q(Q_T)} \leq $$ $$
	\leq c T^{\frac{\varepsilon}{2}} \norm{v}_{\cK_q(Q_T)} \norm{V}_{\cK^{s_q+\varepsilon}_{q}(Q_T)}.
\end{equation}
That $L$ preserves the decay properties near the initial time follows from examining the limit $T \dto 0$ in the estimates above.
One may also show that $B$ is bounded on $\mathring{\cK}_q(Q_T)$ with norm independent of $T$.
Finally, due to \eqref{eq:initialdatabelongs}, we have the property
\begin{equation}\label{u0where}
	\lim_{T \dto 0} \norm{e^{t\Delta} U_0}_{\cK_q(Q_T)} = 0.
\end{equation}
Hence, Lemma \ref{abstract} implies the existence of a mild solution $W \in X_{S}$ to the integral equation \eqref{integraleqn} for a time $0 < S \ll 1$. The solution $W(\cdot,t)$ is also continuous in $L^q(\R^3)$ after the initial time. As for uniqueness, note that each mild solution $W \in \mathring{\cK}_q(Q_T)$ of \eqref{integraleqn} obeys the property that $W + V \in \mathring{\cK}_q(Q_T)$ is a mild solution of the Navier-Stokes equations, so $W + V \equiv NS(u_0)$ and $W \equiv \tilde{U}$ on $Q_T$.

\textit{3. $U$ is in $\mathring{\cK}_q(Q_S)$ for $0 < S \ll 1$.}
Let us return to the approximations $U^\rho$ constructed in the proofs of Theorem \ref{existence} and Proposition \ref{energysol}. Recall that $U^\rho$ solves the system
	\begin{equation}
		\p_t U^\rho - \Delta U^\rho + \div U^\rho \otimes (U^\rho)_\rho + \div U^\rho \otimes V + \div V \otimes U^\rho = - \nabla P^\rho $$ $$
		\div U^\rho = 0,
\end{equation}
with initial condition $U^\rho(\cdot,0) = U_0$.
By repeating the arguments in Step 2, there exists a time $S > 0$ independent of the parameter $\rho > 0$ and a unique mild solution $W^\rho \in \mathring{\cK}_q(Q_S)$ of the integral equation
\begin{equation}\label{integraleqn2}
		W^\rho(\cdot,t) = e^{t \Delta} U_0 - B_\rho(W,W)(\cdot,t) - L(W)(\cdot,t),
	\end{equation}
	where the operator $B_\rho$ is defined by
	\begin{equation}
		B_\rho(v,w)(t) := \int_0^t e^{(t-s)\Delta} \bP \div v \otimes (w)_\rho \, ds.
	\end{equation}
	Since $U_0 \in L^2(\R^3)$, we may shorten the time $S > 0$ so that $W^{\rho} \in L^\infty_t L^2_x(Q_S)$. This assertion follows from applying Lemma \ref{abstract} in the intersection space $\mathring{\cK}_q(Q_S) \cap L^\infty_t L^2_x(Q_S)$ with $0 < S \ll 1$, once we have the necessary estimates:
	\begin{equation}
		\norm{L(v)}_{L^\infty_t L^2_x(Q_T)} \leq c T^{\frac{\varepsilon}{2}} \norm{v}_{L^\infty_t L^2_x(Q_{T})} \norm{V}_{\cK^{s_q+\varepsilon}_{q}(Q_T)},
	\end{equation}
	\begin{equation}
		\norm{B_\rho(v,w)}_{L^\infty_t L^2_x(Q_T)} \leq c \norm{v}_{L^\infty_t L^2_x(Q_T)} \norm{w}_{\cK_q(Q_T)}.
	\end{equation}
	In fact, we obtain that $W^\rho \in C([0,S];L^2(\R^3) \cap L^2_t \dot H^{1}_x(Q_{S})$. This follows from viewing $W^\rho$ as a solution of the mollified Navier-Stokes equations (no lower order terms) with initial data $U_0 \in L^2(\R^3)$ and forcing term $\div F$, where
	\begin{equation}
		F = -W^\rho \otimes V - V \otimes W^\rho \in L^2(Q_S),
		\label{}
	\end{equation}
	since $V \in \cK^{-1+\varepsilon}_{\infty}(Q_S)$ and $W^\rho \in L^\infty_t L^2_x(Q_S)$.

Finally, we pass to the limit as $\rho \dto 0$. 
The only possible issue is that $\mathring{\cK}_q(Q_S)$ is not closed under weak-star limits, but this will not be a problem because
 Theorem~\ref{abstract} actually gave the uniform bound
\begin{equation}
\norm{U^\rho}_{\cK_q(Q_T)} \leq \kappa(T), \quad 0 < T \leq T_0,
\end{equation}
where $\kappa(T)$ is a non-negative continuous function on the interval $[0,S]$ and $\kappa(0) = 0$. Now by taking the weak-star limit
along a subsequence $\rho_k \dto 0$, we obtain
\begin{equation}
\norm{U}_{\cK_q(Q_T)} \leq \kappa(T), \quad 0 < T \leq S.
\end{equation}
Together with the convergence arguments in Proposition \ref{energysol}, this implies that $U \in \mathring{\cK}_q(Q_S)$ is a mild solution of the integral equation \eqref{integraleqn}. By the uniqueness argument in Step 2, $U \equiv \tilde{U}$ on $Q_S$.

	\textit{4. Agreement until $T^*(u_0)$.}
	So far, we have only shown that $U \equiv \tilde{U}$ on a short time interval $[0,S]$. After the initial time, the mild solutions $NS(u_0)$ and $V$ of the Navier-Stokes equations are in subcritical spaces:
		\begin{equation}
			NS(u_0), V \in C([t_0,\min(T_{\sharp},T^*));L^q(\R^3))
			\label{}
		\end{equation}
		for all $0 < t_0 < \min(T_{\sharp},T^*)$. We will prove that $\tilde{U}$ is in the energy space and then argue via weak-strong uniqueness. To do so, one may develop the existence theory for mild solutions $W \in C([t_0,T];L^q(\R^3))$ of the integral equation
		\begin{equation}
			W(\cdot,t) = e^{(t-t_0) \Delta} \tilde{U}(\cdot,t_0) - \int_{t_0}^{t} e^{(t-s)\Delta} \bP \div W \otimes W \, ds $$ $$
			- \int_{t_0}^{t} e^{(t-s)\Delta} \bP \div (W \otimes V + V \otimes W) \, ds,
	\end{equation}
	where $t_0 \in [T_0,\min(T_{\sharp},T^*))$.
	Note that whenever $\tilde{U}(\cdot,t_0) \in L^2(\R^3)$, the unique mild solution $W$ will also be in the energy space. The proof is similar to Steps 2-3. In this way, one obtains that
	\begin{equation}
		\tilde{U} \in C([0,T);L^2(\R^3)) \cap L^2_t \dot H^1_x(Q_T)
		\label{}
	\end{equation}
	for all $0 < T < \min(T_\sharp,T^*)$
	and satisfies the energy equality
	\begin{equation}
		\int_{\R^3} |\tilde{U}(x,t_2)|^2 \, dx + 2 \int_{t_1}^{t_2} \int_{\R^3} |\nabla \tilde{U}|^2 \, dx \,dt =$$ $$= \int_{\R^3} |\tilde{U}(x,t_1)|^2 \,dx + 2 \int_{t_1}^{t_2} \int_{\R^3} V \otimes \tilde{U} : \nabla \tilde{U} \, dx\, dt,
	\end{equation}
	for all $0 \leq t_1 < t_2 < \min(T_{\sharp},T^*)$. In fact, by the Gronwall-type argument in Lemma~\ref{stokes}, one may actually show that $\tilde{U} \in L^2_t \dot H^{1}_x(Q_{\min(T_\sharp,T^*)})$.
In addition, we know that $U$ obeys the global energy inequality \eqref{globalenergy} for almost every $0 \leq t_1 < T_\sharp$, including $t_1 = 0$, and for every $t_2 \in (t_1,T_\sharp]$. Let us write $D := U - \tilde{U}$. Then $D$ obeys the energy inequality
\begin{equation}
	\int_{\R^3} |D(x,t_2)|^2 \, dx + 2 \int_{t_1}^{t_2} \int_{\R^3} |\nabla D|^2 \, dx \, dt \leq $$ $$ \leq 2 \int_{t_1}^{t_2} \int_{\R^3} (\tilde{U} + V) \otimes D : \nabla D \, dx \, dt,
\end{equation}
where $t_1 := T_0/2$. To obtain the energy inequality for $D$, one must write $|D|^2 = |U|^2 + |\tilde{U}|^2 - 2U \cdot \tilde{U}$, $|\nabla D|^2 = |\nabla U|^2 + |\nabla \tilde{U}|^2 - 2 \nabla U : \nabla \tilde{U}$, and utilize the identity
	\begin{equation}\label{eq:weakstrongid}
	\int_{\R^3} U(x,t) \cdot \tilde{U}(x,t) \, dx + 2 \int_0^t \int_{\R^3} \nabla U : \nabla \tilde{U} \, dx \, dt' - \int_{\R^3} U(x,0) \cdot W(x,0) \, dx = $$ $$ = \int_0^t \int_{\R^3} (\p_t U - \Delta U) \cdot \tilde{U} + U \cdot (\p_t \tilde{U} - \Delta \tilde{U}) \, dx \, dt'.
\end{equation}
This argument is typical in weak-strong uniqueness proofs. The identity \eqref{eq:weakstrongid} is clear for smooth vector fields with compact support, but it may be applied in more general situations by approximation.
As in Lemma \ref{stokes}, the energy inequality gives
\begin{equation}
	|D|_{2,Q_{t_1,t_2}}^2 \leq B(t_1,t_2) |D|^2_{2,Q_{t_1,t_2}},
	\label{}
\end{equation}
where we have defined
\begin{equation}
	B(t_1,t_2) := c (\norm{\tilde{U}}_{L^l_t L^q_x(Q_{t_1,t_2})} + \norm{V}_{L^l_t L^q_x(Q_{t_1,t_2})}).
	\label{}
\end{equation}
One then takes $t_2$ close to $t_1$ such that $B(t_1,t_2) \leq 1/2$ to obtain $D \equiv 0$ on $Q_{t_1,t_2}$. The equality $U \equiv \tilde{U}$ may be propagated forward until $\min(T_{\sharp},T^*)$ by repeating the argument (even if $\norm{\tilde{U}}_{L^l_t L^q_x(Q_{t_1,T^*})} = \infty$). This completes the proof.
\end{proof}

Finally, here is a result that contains the limiting arguments we will use in Theorem \ref{main}. Namely, we demonstrate that a weakly converging sequence of initial data has a corresponding subsequence of Calder{\'o}n solutions that converges locally strongly to a solution of the Navier-Stokes equations.

As a reminder, we do not prove that the limit solution is a Calder{\'o}n solution, though we expect such a result to be true. The issue is that the limit solution does not evidently satisfy the energy inequality starting from the initial time.

\begin{thm}[Weak convergence of Calder{\'o}n solutions]\label{weakconv}
	Let $3 < p < \infty$ and $T>0$. Suppose $(u_0^{(n)})_{n \in \N}$ is a sequence of divergence-free vector fields such that
	\begin{equation}
		u_0^{(n)} \wto u_0 \quad \text{ in } \dot B^{s_p}_{p,p}(\R^3).
		\label{}
	\end{equation}
	Then for each $n \in \N$, there exists a Calder{\'o}n solution $u^{(n)}$ on $Q_T$ with initial data $u_0^{(n)}$ and associated pressure $p^{(n)}$ such that the solution $u^{(n)}$ agrees with the mild solution $NS(u_0^{(n)})$ until time $\min(T,T^*(u_0^{(n)}))$.
	
	Furthermore, there exists a distributional solution $(u,p)$ of the Navier-Stokes equations on $Q_T$ with the following properties:
\begin{equation}
u_0 = U_0 + V_0, \quad u = U + V,
\end{equation}
where
\begin{equation}
	U_0 \in L^2(\R^3), \quad V_0 \in \dot B^{s_q+\varepsilon}_{q,q}(\R^3),
	\label{}
\end{equation}
\begin{equation}
	U \in L^\infty_t L^2_x \cap L^2_t \dot H^1_x(Q_T), \quad V \in \cK^{s_q+\varepsilon}_q(Q_T),
	\label{}
\end{equation}
\begin{equation}
	q > p, \quad 0 < \varepsilon < -s_q,
	\label{}
\end{equation}
$V$ is a mild solution of the Navier-Stokes equations on $Q_T$ with initial data $V_0$, and $U$ solves the perturbed Navier-Stokes system \eqref{nselot} in $Q_T$. The vector field $U(\cdot,t)$ is weakly continuous as an $L^2(\R^3)$-valued function on $[0,T]$ and satisfies $U(\cdot,0) = U_0 \in L^2(\R^3)$. The solution $(u,p)$ is suitable in $Q_T$ in the sense of \eqref{localenergyineq}, and $U$ satisfies the local energy inequality \eqref{suitability} in $Q_T$ and the global energy inequality \eqref{globalenergyineq} for almost every $0 < t_1 < T$ and for all $t_2 \in (t_1,T]$. 

Finally, there exists a subsequence, still denoted by $n$, such that the Calder{\'o}n solutions $u^{(n)}$ converge to $u$ in the following senses:
\begin{equation}
u^{(n)} \to u \text{ in } L^3_\loc(\R^3 \times ]0,T]), \quad
		p^{(n)} \wto p \text{ in } L^{3/2}_\loc(\R^3 \times ]0,T]),
\end{equation}
\begin{equation}\label{eq:contconv}
	u^{(n)} \to u \quad \text{ in } C([0,T];\cS'(\R^3)).
\end{equation}
In particular,
\begin{equation}\label{ptwiseest}
	u^{(n)}(\cdot,t) \wstar u(\cdot,t) \quad \text{ in } \cS'(\R^3), \quad 0 \leq t \leq T.
\end{equation}
\end{thm}

\begin{proof}
	\textit{1. Splitting.}
	Let us assume the hypotheses of the theorem. There exists a constant $A>0$ such that
	\begin{equation}
		\norm{u_0^{(n)}}_{\dot B^{s_p}_{p,p}(\R^3)} \leq A.
		\label{}
	\end{equation}
	We will follow the proof of Theorem \ref{existence}. Let $q \in (p,\infty)$. The constants below are independent of $n$ but may depend on $p,q,T$. According to Lemma \ref{splitdata}, there exists $0 < \varepsilon < -s_q$ such that for all $n \in \N$,
	\begin{equation}
		u_0^{(n)} = U_0^{(n)} + V_0^{(n)},
	\end{equation}
	where
	\begin{equation}
		U_0^{(n)} \in L^2(\R^3), \quad \norm{U_0^{(n)}}_{L^2(\R^3)} \leq C(A,M),
		\label{}
	\end{equation}
	\begin{equation}
		V_0^{(n)} \in \dot B^{s_q+\varepsilon}_{q,q}(\R^3), \quad \norm{V_0^{(n)}}_{\dot B^{s_q+\varepsilon}_{q,q}(\R^3)} < M,
		\label{}
	\end{equation}
	According to Theorem \ref{subcrittheory}, we may choose $0 < M \ll 1$ such that the mild solutions $V^{(n)}$ with initial data $V_0^{(n)}$ exist on $Q_T$ and satisfy
	\begin{equation}\label{mildest}
		\norm{t^{k+\frac{l}{2}} \p_t^k \nabla^l V^{(n)}}_{\cK^{s_q+\varepsilon}_{q}(Q_{T})} \leq c(k,l),
\end{equation}
for all integers $k, l \geq 0$.

\textit{2. Convergence of $V^{(n)}$.}
Due to the estimate \eqref{mildest} and the Ascoli-Arzela theorem, we may pass to a subsequence, still denoted by $n$, such that
	\begin{equation}
		V^{(n)} \wstar V\text{ in } \cK^{s_q+\varepsilon}_{q}(Q_{T}), \quad
		\p_t^k \nabla^l V^{(n)} \to \p_t^k \nabla^l V \text{ in } C(K),
	\end{equation}
	for all $K \subset \R^3 \times ]0,T]$ compact and integers $k,l \geq 0$. The Navier-Stokes equations imply
	\begin{equation}\label{rhsv}
	\p_t V^{(n)} = -\Delta V^{(n)} - \bP \div V^{(n)} \otimes V^{(n)}
\end{equation}
in the sense of tempered distributions on $Q_T$, and one may estimate the time derivative $\p_t V^{(n)}$ by the right-hand side of \eqref{rhsv}:
\begin{equation}
	\sup_{n \in \N} \norm{\Delta V^{(n)}}_{L^l_t W^{-2,q}_x(Q_{T})} + \norm{\bP \div V^{(n)} \otimes V^{(n)}}_{L^{\frac{l}{2}}_t W^{-1,\frac{q}{2}}_x(Q_{T})} < \infty,
\end{equation}
where $2/l+3/q=1$. Hence, there exists a subsequence such that
\begin{equation}
	\int_{\R^3} V^{(n)}(x,\cdot) \cdot \varphi \, dx \to \int_{\R^3} V(x,\cdot) \cdot \varphi \,dx \quad  \text{ in } C([0,T]), \quad \varphi \in \cS(\R^3).
\end{equation}
By the Calder{\'o}n-Zygmund estimates and \eqref{mildest}, the associated pressures satisfy
\begin{equation}
	Q^{(n)} := (-\Delta)^{-1} \div \div V^{(n)} \otimes V^{(n)} \wto Q \text{ in } L^{l/2}_t L^{q/2}_x(Q_T).
	\label{}
\end{equation}
In particular, the convergence occurs weakly in $L^{3/2}_\loc(Q_T)$. By similar arguments as in Step 1 of Proposition \ref{energysol}, the pressure satisfies $Q = (-\Delta)^{-1} \div \div V \otimes V$,
and $(V,Q)$ solves the Navier-Stokes equations on $Q_T$ in the sense of distributions. Therefore, since $V \in \cK^{s_q+\varepsilon}_{q}(Q_{T})$ and $V(\cdot,t) \wstar V_0$ in tempered distributions as $t \dto 0$, we conclude that $V$ is the mild solution of the Navier-Stokes equations on $Q_T$ with initial data $V_0$ as in Theorem \ref{subcrittheory}. See \cite{lemarie1}, p. 122, for further remarks on the equivalence between differential and integral forms of the Navier-Stokes equations.

\textit{3. Convergence of $U^{(n)}$.} For each $n \in \N$, let $U^{(n)} \in L^\infty_t L^2_x \cap L^2_t \dot H^1_x(Q_T)$, whose existence is guaranteed by Theorem \ref{mildcalderon}, so that the Calder{\'o}n solution
	\begin{equation}
		u^{(n)} = U^{(n)} + V^{(n)}
	\end{equation}
	agrees with the mild solution $NS(u_0^{(n)})$ up to $\min(T,T^*(u_0^{(n)}))$. It remains to consider the limit of $U^{(n)}$. The proof is very similar to the proofs of Proposition \ref{energysol} and Theorem \ref{existence}, so we will merely summarize what must be done. Recall from Step 1 that $\norm{U_0^{(n)}}_{L^2(\R^3)} \leq C(A,M)$. We use the energy inequality and a Gronwall-type argument to obtain
\begin{equation}
	|U^{(n)}|_{2,Q_T}^2 \leq C(A,M)
\end{equation}
Hence, we may take weak-star limits in the energy space upon passing to a subsequence. As before, we may estimate the time derivative using the perturbed Navier-Stokes equations. The result is that
\begin{equation}
	\norm{\p_t U^{(n)}}_{L^2_t H^{-3/2}_x(Q_{T})} \leq C(A,M).
	\label{}
\end{equation}
Consequently, we may extract a further subsequence such that $U^{(n)} \to U$ in $L^3_t (L^3_\loc)_x(Q_T)$, the limit $U(\cdot,t)$ is weakly continuous on $[0,T]$ as an $L^2(\R^3)$-valued function, and
\begin{equation}
	\int_{\R^3} U^{(n)}(x,\cdot) \cdot \varphi(x) \, dx \to \int_{\R^3} U(x,\cdot) \cdot \varphi(x) \, dx \text{ in } C([0,T]), \quad \varphi \in L^2(\R^3).
\end{equation}
Now recall the associated pressures $P^{(n)} := P_1^{(n)} + P_2^{(n)}$,
\begin{equation}
	P_1^{(n)} := (-\Delta)^{-1} \div \div U^{(n)} \otimes U^{(n)},
	\label{}
\end{equation}
\begin{equation}
	P_2^{(n)} := (-\Delta)^{-1} \div \div (U^{(n)} \otimes V^{(n)} + V^{(n)} \otimes U^{(n)}).
	\label{}
\end{equation}
By the Calder{\'o}n-Zygmund estimates, we pass to a subsequence to obtain
\begin{equation}
	P_1^{(n)} \wto P_1 \text{ in } L^2_t L^{3/2}_x(Q_T), \quad
	P_2^{(n)} \wto P_2 \text{ in } L^2(Q_T),
	\label{}
\end{equation}
\begin{equation}
	P^{(n)} \wto P := P_1 + P_2 \text{ in } L^{3/2}_\loc(Q_T).
\end{equation}
Since $-\Delta P_1 = \div \div U \otimes U$ and $-\Delta P_2 = \div \div (U \otimes V + V \otimes U)$ in the sense of distributions on $Q_T$, we obtain from the Liouville theorem that
\begin{equation}
	P_1 = (-\Delta)^{-1} \div \div U \otimes U, \quad P_2 = (-\Delta)^{-1} \div \div (U \otimes V + V \otimes U).
	\label{}
\end{equation}
Define $p := P + Q$. It is clear from the limiting procedure that $(u,p)$ satisfies the Navier-Stokes equations on $Q_T$ and $(U,P)$ satisfies the perturbed Navier-Stokes equations on $Q_T$. Each is satisfied in the sense of distributions.

As in the proofs of Step 2 in Proposition \ref{energysol} and Step 2 in Theorem \ref{existence}, respectively, we may show that the limit $(u,p)$ is a suitable weak solution of the Navier-Stokes equations on $Q_T$ in the sense of \eqref{localenergyineq} and that $(U,P)$ satisfies the local energy inequality \eqref{suitability} on $Q_T$. Moreover, as in Step 3 of Proposition \ref{energysol}, \eqref{suitability} implies the global energy inequality \eqref{globalenergyineq} for almost every $0 < t_1 < T$ and for all $t_2 \in (t_1,T]$. This completes the proof.
\end{proof}

\section{Proof of main results}

We are ready to prove Theorem \ref{main}. The proof follows the scheme set forth in \cite{sereginl3} except that we use Calder{\'o}n solutions to take the limit of the rescaled solutions.
\begin{proof}[Proof of Theorem \ref{main}]
	Let us assume the hypotheses of the theorem. According to the chain of embeddings \eqref{chain}, we have
	\begin{equation}
		u_0 \in \dot B^{s_p}_{p,q}(\R^3) \into \dot B^{s_m}_{m,m}(\R^3), \quad  m := \max(p,q).
		\label{}
	\end{equation}
	By the uniqueness results in Theorem \ref{mildexist}, the notion of mild solution and maximal time of existence are unchanged by considering the larger homogeneous Besov space. Thus, without loss of generality, we will assume $p=q=m$.

	1. \textit{Rescaling}.
	Let $u$ be the mild solution of the Navier-Stokes equations with divergence-free initial data $u_0 \in \dot B^{s_p}_{p,p}(\R^3)$ and $T^*(u_0) < \infty$ from the statement of the theorem. In Corollary \ref{singularity}, we proved that $u$ must form a singularity at time $T^*(u_0)$. By the translation and scaling symmetries of the Navier-Stokes equations, we may assume that the singularity occurs at the spatial origin and time $T^*(u_0)=1$.
	
	Suppose for contradiction that there exists a sequence $t_n \upto 1$ and constant $M > 0$ such that
	\begin{equation}
		\norm{u(\cdot,t_n)}_{\dot B^{s_p}_{p,p}(\R^3)} \leq M.
		\label{}
	\end{equation}
	The solution $u(\cdot,t)$ is continuous on $[0,1]$ in the sense of tempered distributions (for instance, because $u$ agrees with a Calder{\'o}n solution on $Q_T$). We must have
	\begin{equation}
		\norm{u(\cdot,1)}_{\dot B^{s_p}_{p,p}(\R^3)} \leq M.
		\label{}
	\end{equation}
Let us zoom in around the singularity. For each $n \in \N$, we define
	\begin{equation}
		u^{(n)}(x,t) := \lambda_n u(\lambda_n x, t_n + \lambda_n^2 t), \quad (x,t) \in Q_1,
	\end{equation}
	where $\lambda_n := (1-t_n)^{1/2}$. Then $u^{(n)}$ is the mild solution of the Navier-Stokes equations on $Q_1$ with divergence-free initial data $u_0^{(n)} = \lambda_n u(\lambda_n x,t_n)$, and
	\begin{equation}
		\norm{u_0^{(n)}}_{\dot B^{s_p}_{p,p}(\R^3)} \leq M.
		\label{}
	\end{equation}
Let us pass to a subsequence, still denoted by $n$, such that
\begin{equation}
	u_0^{(n)} \wto v_0 \text{ in } \dot B^{s_p}_{p,p}(\R^3).
	\label{}
\end{equation}

	2. \textit{Limiting procedure}.
	We now apply Theorem \ref{weakconv} to the weakly converging sequence $(u_0^{(n)})_{n \in \N}$. For each $n \in \N$, there exists a Calder{\'o}n solution $u^{(n)}$ on $Q_1$ with initial data $u_0^{(n)}$ such that $u^{(n)}$ agrees with the mild solution $NS(u_0^{(n)})$ on $Q_1$. Furthermore, by passing to a subsequence, still denoted by $n$, we have
	\begin{equation}
		u^{(n)} \to v \text{ in } L^3_\loc(Q_1), \quad p^{(n)} \wto q \text{ in } L^{3/2}_\loc(Q_1),
		\label{}
	\end{equation}
	where $(v,q)$ solves the Navier-Stokes equations in the sense of distributions on $Q_1$ and satisfies the many additional properties listed in Theorem \ref{weakconv}. In particular, $v \in C([0,1];\cS'(\R^3))$. According to Lemma \ref{persist} concerning persistence of the singularity, the solution $v$ also has a singularity at the spatial origin and time $T=1$.

	Next, we observe that the solution $v$ vanishes identically at time $T=1$:
	\begin{equation}\label{vvanish}
		v(\cdot,1) = 0.
	\end{equation}
	Indeed, Theorem \ref{weakconv} implies that
	\begin{equation}
		u^{(n)}(\cdot,1) \wstar v(\cdot,1) \text{ in } \cS'(\R^3),
	\end{equation}
	and by the scaling property of $u(\cdot,1) \in \dot B^{s_p}_{p,p}(\R^3)$,
	\begin{equation}\label{rescalevanish}
		\la u^{(n)}(\cdot,1), \varphi \ra = \la u(\cdot,1), \lambda_n^{-2} \varphi(\cdot/\lambda_n) \ra \to 0, \quad \varphi \in \cS(\R^3).
	\end{equation}
The property \eqref{rescalevanish} is a consequence of the density of Schwartz functions in $\dot B^{s_p}_{p,p}(\R^3)$. It is certainly true with $u(\cdot,1)$ replaced by a Schwartz function $\psi \in \cS(\R^3)$, and therefore,
	\begin{equation}
		|\la u(\cdot,1), \lambda_n^{-2} \varphi(\cdot/\lambda_n) \ra| \leq |\la \psi, \lambda_n^{-2} \varphi(\cdot/\lambda_n) \ra| + |\la u(\cdot,1) - \psi, \lambda_n^{-2} \varphi(\cdot/\lambda_n)| \leq $$ $$
		\leq o(1) + c \norm{u(\cdot,1) - \psi}_{\dot B^{s_p}_{p,p}(\R^3)} \norm{\varphi}_{\dot B^{-s_p}_{p',p'}(\R^3)} \text{ as } n \to \infty,
	\end{equation}
	where $1/p+1/p' = 1$, for all $\varphi \in \cS(\R^3)$.

	3. \textit{Backward uniqueness}.
	Our goal is to demonstrate that
	\begin{equation}\label{goal}
		\omega := \curl v \equiv 0 \text{ on } Q_{1/2,1}.
	\end{equation}
	Suppose temporarily that \eqref{goal} is satisfied.
	From the well-known vector identity $\Delta v = \grad \div v - \curl \curl v$, we obtain
	\begin{equation}
		\Delta v = 0 \text{ on } Q_{1/2,1}.
		\label{}
	\end{equation}
	Now the Liouville theorem for entire harmonic functions and the decomposition
	\begin{equation}\label{eq:vregularity}
		v=U+V \in L^3(Q_1) + L^\infty_t L^q_x(Q_{\delta,1}), \quad 0 < \delta < 1,
	\end{equation}
	from Theorem \ref{weakconv}
	imply that $v \equiv 0 \text{ on }Q_{1/2,1}$.
	This contradicts that $v$ is singular at time $T=1$ and finishes the proof.

	We will now prove \eqref{goal}. Based on \eqref{eq:vregularity} and
	\begin{equation}
		q=P+Q \in L^{3/2}(Q_1) + L^\infty_t L^{q/2}_x(Q_{\delta,1}) , \quad 0 < \delta < 1,
		\label{}
	\end{equation}
	we have
\begin{equation}
	\lim_{|x| \to \infty} \int_{1/4}^{1} \int_{B(x,1)} |v|^3 + |q|^{3/2} \, dx \,dt = 0.
	\label{}
\end{equation}
Recall from Theorem \ref{weakconv} that $(v,q)$ obeys the local energy inequality \eqref{localenergyineq}, so
it is a suitable weak solution on $Q_1$. The $\varepsilon$-regularity criterion in Theorem \ref{ckncrit} implies that there exist constants $R, \kappa > 0$, $K := (\R^3 \setminus B(R)) \times (1/2,1)$, such that
\begin{equation}
	\sup_K |v| + |\nabla v| + |\nabla^2 v| < \kappa.
\end{equation}
Recall the equation satisfied by the vorticity: $\p_t \omega - \Delta \omega = -\curl (v \nabla v)$. This implies
	\begin{equation}\label{diffineq}
		|\p_t \omega - \Delta \omega| \leq c (|\nabla \omega| + |\omega|) \quad \text{ in } K
\end{equation}
for a constant $c>0$ depending on $\kappa$. Also, $w(\cdot,1) = 0$ due to \eqref{vvanish}.
Now, according to Theorem \ref{backuniqueness} concerning backward uniqueness, $\omega \equiv 0$ in $K$.

It remains to demonstrate that $\omega \equiv 0$ in $\overline{B(R)} \times (1/2,1)$. We claim that there exists a dense open set $G \subset (0,1)$ such that $v$ is smooth on $\Omega := \R^3 \times G$.
With the claim in hand, let $z_0 = (x_0,t_0) \in \Omega \cap K$ such that $|x_0| = 2R$. Note that $\omega \equiv 0$ in a neighborhood of $z_0$. In addition, by the smoothness of $v$, there exist $0 < \varepsilon \ll 1$ and $c>0$ depending on $z_0$ such that
\begin{equation}
	|\p_t \omega - \Delta \omega| \leq c(|\nabla \omega| + |\omega|) \text{ in } Q := B(x_0,4R) \times (t_0-\varepsilon,t_0+\varepsilon) \subset \Omega.
	\label{}
\end{equation}
Hence, the assumptions of Theorem \ref{continuation} concerning unique continuation are satisfied in $Q$, and $\omega \equiv 0$ in $Q$. This implies that $\omega \equiv 0$ in $\R^3 \times (t_0-\varepsilon,t_0+\varepsilon)$. Since $z_0 \in \Omega \cap K$ was arbitrary, we obtain that $\omega \equiv 0$ in $\Omega$. Now the density of $G$ and weak continuity $v \in C([0,1];\cS'(\R^3))$ imply $\omega \equiv 0$ on $Q_{1/2,1}$.

Finally, we prove the claim that there exists a dense open set $G \subset (0,1)$ such that $v$ is smooth on $\R^3 \times G$. Recall from Theorem \ref{weakconv} that $v = U + V$, where $V$ is the smooth mild solution in $Q_1$ from Theorem \ref{subcrittheory} with initial data $V_0$. To treat $U$, consider the set $\Pi$ of times $t_1 \in (0,1)$ such that $U(\cdot,t_1) \in H^1(\R^3)$ and $U$ satisfies the global energy inequality \eqref{globalenergyineq} for all $t_2 \in (t_1,1]$. The latter condition ensures that 
\begin{equation}
	\lim_{t \dto t_1} \norm{U(\cdot,t) - U(\cdot,t_1)}_{L^2(\R^3)} = 0.
	\label{}
\end{equation} The set $\Pi$ has full measure in $(0,1)$.
For each $t_1 \in \Pi$, there exist $\varepsilon := \varepsilon(t_1) > 0$, a vector field $\tilde{U} \in C([t_1,t_1+\varepsilon];H^1(\R^3))$, and pressure $\tilde{P} \in L^\infty_t L^3_x(Q_{t_1,t_1+\varepsilon})$ such that
\begin{equation}\label{Utildeobey}
	\p_t \tilde{U} - \Delta \tilde{U} + \div \tilde{U} \otimes \tilde{U} + \div V \otimes \tilde{U} + \div \tilde{U} \otimes V = - \nabla \tilde{P}, \quad \div \tilde{U} = 0,
\end{equation}
in the sense of distributions on $Q_{t_1,t_1+\varepsilon}$. Furthermore, $\tilde{U}$ is smooth on $Q_{t_1,t_1+\varepsilon}$. This may be proven by developing the local well-posedness theory for the integral equation
\begin{equation}
	\tilde{U}(\cdot,t) = e^{(t-t_1)\Delta} U(\cdot,t_1) - \int_{t_1}^t e^{(t-s) \Delta} \bP \div (\tilde{U} \otimes \tilde{U} + V \otimes \tilde{U} + \tilde{U} \otimes V) \, ds
	\label{}
\end{equation}
with $U(\cdot,t_1) \in H^1(\R^3)$. One must use that $\p_t^k \nabla^l V \in C([\delta,1];L^\infty(\R^3))$ for all $0 < \delta < 1$ and integers $k,l \geq 0$. By weak-strong uniqueness for the equation \eqref{Utildeobey}, as in Step 4 in Theorem \ref{mildcalderon}, we must have that $U \equiv \tilde{U}$ on $Q_{t_1,t_1+\varepsilon}$. Hence, $U$ is smooth on $Q_{t_1,t_1+\varepsilon}$. We define the dense open set $G \subset (0,1)$ as follows:
\begin{equation}
	G := \bigcup_{t_1 \in \Pi} (t_1,t_1+\varepsilon(t_1)), \quad \overline{G} \supseteq \overline{\Pi} = [0,1],
	\label{}
\end{equation}
This completes the proof of the claim and Theorem \ref{main}.

	
\end{proof}

\section{Appendix}
This appendix has been written with two goals in mind.

The first goal is to unify two well-known approaches to mild solutions with initial data in the critical homogeneous Besov spaces. Mild solutions in the time-space homogeneous Besov spaces were considered in the blow-up criterion of Gallagher-Koch-Planchon \cite{koch}, whereas mild solutions in Kato spaces are better suited to our own needs. In Theorem \ref{mildexist}, we prove that the two notions of solution coincide and have the same maximal time of existence. Section 4.3 contains the proof, while Sections 4.1--4.2 are devoted to background material on Littlewood-Paley theory and Besov spaces. We also address the subcritical theory in Kato spaces in Theorem \ref{subcrittheory}. Our hope is that the details provided herein will equip the reader to fill in any fixed point arguments we have omitted in Sections 2-3.

The second goal is to collect various results related to the theory of suitable weak solutions. While the blow-up arguments we have presented are not overly complicated, they do rely on technical machinery developed by a number of authors starting in the 1970s. We summarize the theory in Section 4.4 and refer to Escauriaza-Seregin-{\u S}ver{\'a}k \cite{sverak03} for additional details.

\subsection{Littlewood-Paley theory and homogeneous Besov spaces}


We will now summarize the basics of Littlewood-Paley theory and homogeneous Besov spaces. Our treatment is based on the presentation in \cite{bahourichemin}, Chapter 2. The situation is as follows. There exist smooth functions $\varphi$ and $\chi$ on $\R^3$ with the properties
\begin{equation}
	\supp(\varphi) \subset \{ \xi \in \R^3 : 3/4 \leq |\xi| \leq 8/3 \}, \qquad \sum_{j \in \Z} \varphi(2^{-j} \xi) = 1, \, \xi \in \R^3 \setminus \{ 0 \},
\end{equation}
\begin{equation}
	\supp(\chi) \subset \{ \xi \in \R^3 : |\xi| \leq 4/3 \}, \qquad \chi(\xi) + \sum_{j \geq 0} \varphi(2^{-j} \xi) = 1, \, \xi \in \R^3.
\end{equation}
For all $j \in \Z$, we define the homogeneous dyadic block $\dot \Delta_j$ and the homogeneous low-frequency cutoff $\dot S_j$ to be the following Fourier multipliers:
\begin{equation}\label{eq:dyadicblockdef}
	\dot \Delta_j := \varphi(2^{-j} D), \quad \dot S_j := \chi(2^{-j} D).
\end{equation}
For tempered distributions $u_0$ on $\R^3$, the convergence of the sum $\sum_{j \leq 0} \dot \Delta_j u_0$ typically occurs only in the sense of tempered distributions modulo polynomials (see p. 28-30 in \cite{lemarie1}). To remove ambiguity, we will consider the following subspace of tempered distributions on $\R^3$:
\begin{equation}\label{eq:tempereddecay}
	\cS_h' := \big\lbrace u_0 \in \cS' : \lim_{\lambda \to \infty} \norm{\theta(\lambda D) u_0}_{L^\infty(\R^3)} = 0 \text{ for all } \theta \in \cD(\R^3) \big\rbrace.
\end{equation}
The subspace $\cS_h'$ is not closed in the standard topology on tempered distributions. We will often refer to the condition defining \eqref{eq:tempereddecay} as the ``realization condition.'' Let us recall the family of homogeneous Besov seminorms, defined for all tempered distributions $u_0$ on $\R^3$:
\begin{equation}\label{eq:besovnorm}
	\norm{u_0}_{\dot B^{s}_{p,q}(\R^3)} := \big\lVert 2^{js} \norm{\dot \Delta_j u_0}_{L^p(\R^3)} \big\rVert_{\ell^q(\Z)}, \quad s \in \R, \, 1 \leq p,q \leq \infty.
\end{equation}
These are norms when restricted to tempered distributions in the class $u_0 \in \cS'_h$. We now introduce the family of homogeneous Besov spaces,
\begin{equation}\label{eq:besovdef}
	\dot B^s_{p,q}(\R^3) := \big\lbrace u_0 \in \cS_h' : \norm{u_0}_{\dot B^{s}_{p,q}(\R^3)} < \infty \big\rbrace, \quad s \in \R, \, 1 \leq p, q \leq \infty.
\end{equation}
This is a family of normed vector spaces. As long as the condition
\begin{equation}\label{eq:indexrestriction}
	s < 3/p \quad \text{ or } \quad s=3/p, \, q=1
\end{equation}
is satisfied, $\dot B^{s}_{p,q} \cap \dot B^{s_1}_{p_1,q_1}(\R^3)$ is a Banach space for all $s_1 \in \R$ and $1 \leq p_1, q_1 \leq \infty$, and there is no ambiguity modulo polynomials.

Let us now recall a particularly useful characterization of homogeneous Besov spaces in terms of Kato-type norms. Our reference is \cite{bahourichemin} Theorem 2.34. Let $0 < T \leq \infty$. The following family of norms is defined for locally integrable functions $u \in L^1_\loc(Q_T)$:
\begin{equation}
	\norm{u}_{\cK^s_{p,q}}(Q_T) := \big\lVert t^{-\frac{s}{2}} \norm{u(\cdot,t)}_{L^p(\R^3)} \big\rVert_{L^q((0,T),\frac{dt}{t})}, \quad s \in \R, \, 1 \leq p, q \leq \infty.
	\label{}
\end{equation}
We now define the Kato spaces with the above norms:
\begin{equation}\label{katospacedef}
	\cK^s_{p,q}(Q_T) := \left\lbrace u \in L^1_\loc(Q_T) : \norm{u}_{\cK^s_{p,q}(Q_T)} < \infty \right\rbrace.
\end{equation}
To simplify notation, we will denote
\begin{equation}
	\cK^s_p(Q_T) := \cK^s_{p,\infty}(Q_T), \quad \cK_p(Q_T) := \cK^{s_p}_p(Q_T).
	\label{}
\end{equation}
The caloric characterization of homogeneous Besov spaces is as follows.
For all $s<0$, there exists a constant $c := c(s) > 0$ such that
\begin{equation}\label{katoendup}
	c^{-1} \norm{e^{t\Delta} u_0}_{\cK^{s}_{p,q}(Q_\infty)} \leq \norm{u_0}_{\dot B^{s}_{p,q}(\R^3)} \leq c \norm{e^{t\Delta} u_0}_{\cK^{s}_{p,q}(Q_\infty)}
\end{equation}
for all tempered distributions $u_0$ on $\R^3$.

We now introduce the time-space homogeneous Besov spaces that appear naturally when solving the Navier-Stokes equations with initial data in homogeneous Besov spaces. Our presentation follows \cite{bahourichemin} Section 2.6.3. Let $0 < T \leq \infty$. We have the following family of seminorms on tempered distributions $u \in \cS'(Q_{T})$:
\begin{equation}
	\norm{u}_{\tilde{L}^r_{T} \dot B^{s}_{p,q}} := \big\lVert 2^{js} \norm{\Delta_j u_0}_{L^r_t L^p_x(Q_{T})} \big\rVert_{\ell^q(\Z)}, \quad s \in \R, \, 1 \leq p,q,r \leq \infty.
\end{equation}
The time-space homogeneous Besov spaces on $Q_{T}$ are defined below:
\begin{equation}\label{eq:timespacedef}
	\tilde{L}^r_{T} \dot B^{s}_{p,q} := \big\lbrace u \in \cS'(Q_{T}) : \norm{u}_{\tilde{L}^r_T \dot B^{s}_{p,q}} < \infty, \; \lim_{j \dto -\infty} \norm{\dot S_j u}_{L^1_t L^\infty_x(Q_{T_1,T_2})} = 0,
	$$ $$  \text{for all } 0 < T_1 < T_2 < T \big\rbrace, \qquad
s \in \R, \, 1 \leq p, q,r \leq \infty.
\end{equation}
The second condition in \eqref{eq:timespacedef} is analogous to the realization condition \eqref{eq:tempereddecay}. We have that $\tilde{L}^r_{T} \dot B^{s}_{p,q} \cap\tilde{L}^{r_1}_{T} \dot B^{s_1}_{p_1,q_1}$ is a Banach space for all $s_1 \in \R$ and $1 \leq r_1, p_1, q_1 \leq \infty$ provided that \eqref{eq:indexrestriction} is satisfied. To simplify notation, we will omit the reference to $T$ in the norm when $T=\infty$. We also sometimes employ spaces $\tilde{L}^r_{\delta,T} \dot B^{s}_{p,q}$ on the spacetime domain $Q_{\delta,T}$ that are defined in the same way.

We now review the Bony paraproduct decomposition as described in \cite{bahourichemin} Section 2.6. Consider the operators
\begin{equation}\label{eq:paraproductdef}
	\dot T_u v := \sum_{j \in \Z} \dot S_{j-1} u \, \dot \Delta_j v, \quad
	\dot R(u,v) := \sum_{|j-j'|\leq 1} \dot \Delta_j u \, \dot \Delta_{j'} v,
\end{equation}
defined formally for all tempered distributions $u, v$ on $\R^3$. These operators represent low-high and high-high interactions in the formal product
\begin{equation}\label{eq:paraproduct}
	uv = \dot T_u v + \dot T_v u + \dot R(u,v).
\end{equation}
If the sums defining the paraproduct operators converge, one may use \eqref{eq:paraproduct} to extend the notion of product to a wider class. Consider
\begin{equation}
	1 \leq p,p_1,p_2,q,q_1,q_2 \leq \infty, \quad s, s_1, s_2 \in \R, $$ $$
	\frac{1}{p} = \frac{1}{p_1} + \frac{1}{p_2}, \quad \frac{1}{q} = \frac{1}{q_1} + \frac{1}{q_2}, \quad s = s_1+s_2.
\end{equation}
Then one has the estimates
\begin{equation}
	\norm{\dot T_u v}_{\dot B^{s}_{p,q}(\R^3)} \leq c(s,s_1) \norm{u}_{\dot B^{s_1}_{p_1,q_1}(\R^3)} \norm{v}_{\dot B^{s_2}_{p_2,q_2}(\R^3)}, \quad (\text{provided } s_1 < 0),
	\label{}
\end{equation}
\begin{equation}
	\norm{\dot R(u,v)}_{\dot B^{s}_{p,q}(\R^3)} \leq c(s) \norm{u}_{\dot B^{s_1}_{p_1,q_1}(\R^3)} \norm{v}_{\dot B^{s_2}_{p_2,q_2}(\R^3)}, \quad (\text{provided } s > 0).
	\label{}
\end{equation}
The additional condition \eqref{eq:indexrestriction} will imply that $\dot T_u v, \dot R(u,v) \in \cS'_h$.
The analogous estimates in time-space homogeneous Besov norms are as follows:
\begin{equation}\label{eq:lowhighest}
	\norm{\dot T_u v}_{\tilde L^r_T \dot B^{s}_{p,q}} \leq c(s,s_1) \norm{u}_{\tilde L^{r_1}_T \dot B^{s_1}_{p_1,q_1}} \norm{v}_{\tilde L^{r_2}_T \dot B^{s_2}_{p_2,q_2}}, \quad (\text{provided } s_1 < 0),
\end{equation}
\begin{equation}\label{eq:highhighest}
	\norm{\dot R(u,v)}_{\tilde L^r_T \dot B^{s}_{p,q}} \leq c(s) \norm{u}_{\tilde L^{r_1}_T \dot B^{s_1}_{p_1,q_1}} \norm{v}_{\tilde L^{r_2}_T \dot B^{s_2}_{p_2,q_2}}, \quad (\text{provided } s > 0),
\end{equation}
where
\begin{equation}
	1 \leq r,r_1,r_2 \leq \infty, \quad \frac{1}{r} = \frac{1}{r_1} + \frac{1}{r_2}, \quad 0 < T \leq \infty.
	\label{}
\end{equation}
The paraproduct decomposition will play a crucial role in proving estimates on the nonlinear term.

\subsection{Heat estimates in homogeneous Besov spaces}
In this subsection, we recall estimates for the heat equation in time-space homogeneous Besov spaces and Kato spaces.

Regarding heat estimates of frequency-localized data, the primary observation is the following, see \cite{bahourichemin} Section 2.1.2 as well as the appendix of \cite{koch}. Let $\cC \subset \R^3$ be an annulus and $\lambda > 0$. There exist constants $C,c >0$ depending only on the annulus $\cC$ such that for all tempered distributions $u_0$ satisfying $\supp(\widehat{u_0}) \subset \lambda \cC$,
\begin{equation}\label{eq:freqlocalizedheat}
	\norm{e^{t\Delta} u_0}_{L^p(\R^3)} \leq Ce^{-ct\lambda^2} \norm{u_0}_{L^p(\R^3)}, \quad 1 \leq p \leq \infty,
\end{equation}
Let $0 < T \leq \infty$ and $f$ be a tempered distribution on $Q_T$ with spatial Fourier transform satisfying $\supp{\widehat{f}} \subset \lambda \cC \times [0,T]$. There exists a constant $C  > 0$, depending only on the annulus $\cC$, such that
\begin{equation}
	\norm{e^{t\Delta} u_0}_{L^{r_2}_t L^{p_2}_x(Q_\infty)} \leq C \lambda^{-\frac{2}{r_2}} \lambda^{3 (\frac{1}{p_1} - \frac{1}{p_2})}  \norm{u_0}_{L^{p_1}(\R^3)},
	\label{}
\end{equation}
\begin{equation}
	\norm{\int_0^t e^{(t-s)\Delta} f(\cdot,s) \, ds}_{L^{r_2}_t L^{p_2}_x(Q_T)} \leq C \lambda^{2(\frac{1}{r_1} - \frac{1}{r_2} - 1)} \lambda^{3(\frac{1}{p_1}-\frac{1}{p_2})} \norm{f}_{L^{r_1}_t L^{p_1}_x(Q_T)},
	\label{}
\end{equation}
\begin{equation}
	1 \leq p_1 \leq p_2 \leq \infty, \quad 1 \leq r_1 \leq r_2 \leq \infty.
	\label{}
\end{equation}
The estimates follow from \eqref{eq:freqlocalizedheat} by Bernstein's inequality and Young's convolution inequality. As an application, we obtain
\begin{equation}\label{eq:heatestbesovu0}
	\norm{e^{t\Delta} u_0}_{\tilde{L}^{r_2} \dot B^{s_{p_2}+\frac{2}{r_2}}_{p_2,q_2}} \leq C \norm{u_0}_{\dot B^{s_{p_1}}_{p_1,q_1}(\R^3)},
\end{equation}
\begin{equation}\label{eq:heatestbesovf}
	\norm{\int_0^t e^{(t-s)\Delta} \bP \div F(\cdot,s) \, ds}_{\tilde{L}^{r_2}_T \dot B^{s_{p_2}+\frac{2}{r_2}}_{p_2,q_2}} \leq C \norm{F}_{\tilde{L}^{r_1}_T \dot B^{s_{p_1}+\frac{2}{r_1}-1}_{p_1,q_1}},
\end{equation}
\begin{equation}
	1 \leq p_1 \leq p_2 \leq \infty, \quad 1 \leq q_1 \leq q_2 \leq \infty, \quad 1 \leq r_1 \leq r_2 \leq \infty.
	\label{}
\end{equation}
Here, we have employed that $\bP$ is a homogeneous Fourier multiplier of degree zero smooth away from the origin, see Proposition 2.40 in \cite{bahourichemin}. Let us now comment on continuity in time. Regarding the estimates \eqref{eq:heatestbesovu0} and \eqref{eq:heatestbesovf}, the solutions belong to the class $C([0,T];\dot B^{s_{p_2}}_{p_2,q_2}(\R^3))$ as long as $r_1, q_1 < \infty$ and the realization condition in \eqref{eq:timespacedef} is met. For example, the realization condition is met whenever $(s_{p_2}+\frac{2}{r_2},p_2,q_2)$ satisfies \eqref{eq:indexrestriction}. Because the mild solutions we seek exist in spaces $\tilde{L}^\infty_T \dot B^{s_p}_{p,q}$ with $3 < p < \infty$, realization will often be automatic.

The second set of estimates we will discuss are the estimates in Kato spaces that arise naturally from the caloric characterization \eqref{katoendup} of homogeneous Besov spaces. We summarize them in a single lemma:
\begin{lem}[Estimates in Kato spaces]\label{katoest}
	Let $0 < T \leq \infty$ and $1 \leq p_1 \leq p_2 \leq \infty$ such that
	\begin{equation}\label{eq:Oseenscaling}
		s_2 - \frac{3}{p_2} = 1 + s_1 - \frac{3}{p_1}.
	\end{equation}
	In addition, assume the conditions
	\begin{equation}\label{eq:Oseenconds}
		s_1 > -2, \qquad \frac{3}{p_1} - \frac{3}{p_2} < 1.
	\end{equation}
	(For instance, if $p_2 = \infty$, then the latter condition is satisfied when $p_1 > 3$. If $p_1 = 2$, then the latter condition is satisfied when $p_2 < 6$.) Then
	\begin{equation}\label{eq:Oseenestimate}
	\norm{\int_0^t  e^{(t-\tau) \Delta} \bP \div F(\cdot,\tau) \, d\tau}_{\cK^{s_2}_{p_2}(Q_T)} \leq C(s_1,p_1,p_2) \norm{F}_{\cK^{s_1}_{p_1}(Q_T)},
\end{equation}
for all distributions $F \in \cK^{s_1}_{p_1}(Q_T)$, and the solution $u$ to the corresponding heat equation belongs to $C((0,T];L^{p_2}(\R^3))$. Let $k, l \geq 0$ be integers. If we further require that
\begin{equation}
	t^{\alpha+\frac{|\beta|}{2}} \p_t^\alpha \nabla^\beta  F \in \cK^{s_1}_{p_1}(Q_T),
	\label{}
\end{equation}
for all integers $0 \leq \alpha \leq k$ and multi-indices $\beta \in (\N_0)^3$ with $|\beta| \leq l$,
then we have
\begin{equation}
	\norm{t^{k+\frac{l}{2}} \p_t^k \nabla^l \int_0^t  e^{(t-\tau) \Delta} \bP \div F(\cdot,\tau) \, d\tau}_{\cK^{s_2}_{p_2}(Q_T)} \leq $$ $$
	\leq C(k,l,s_1,p_1,p_2) \Big( \sum_{\alpha=0}^k \sum_{\beta=0}^l \norm{t^{\alpha+\frac{|\beta|}{2}} F}_{\cK^{s_1}_{p_1}(Q_T)} \Big),
\end{equation}
and the spacetime derivatives $\p_t^k \nabla^l u$ of the solution $u$ belong to $C((0,T];L^{p_2}(\R^3))$.
\end{lem}
\begin{proof}
	Recall the following estimates on the Oseen kernel (Chapter 11 of \cite{lemarie1}). Let $\alpha \geq 0$ be an integer and $\beta \in (\N_0)^3$ be a multi-index. Then
	\begin{equation}\label{eq:Oseennormbounds}
		\norm{\p_t^\alpha \nabla^\beta e^{t\Delta} \bP \div u_0}_{L^{p_2}(\R^3)} \leq c(\alpha,\beta) t^{\frac{1}{2} (\frac{3}{p_2} - \frac{3}{p_1}-1-|\beta|-2\alpha)} \norm{u_0}_{L^{p_1}(\R^3)},
	\end{equation}
	for all $t > 0$ and $1 \leq p_1 \leq p_2 \leq \infty$. In addition, the semigroup commutes with partial derivatives in the space variables.
	
	Let us consider the case when $\alpha$, $\beta$ are zero. Suppose that $s_1,s_2,p_1,p_2$, and $F$ obey the hypotheses of the lemma. Then
	\begin{equation}
		\norm{ \int_0^t e^{(t-\tau) \Delta} \bP \div F(\cdot,\tau) \, d\tau}_{L^{p_2}(\R^3)} \leq \int_0^t \norm{e^{(t-\tau) \Delta} \bP \div F(\cdot,\tau)}_{L^{p_2}(\R^3)} \, d\tau \leq $$ $$
		\stackrel{\eqref{eq:Oseennormbounds}}{\leq} c \int_0^t (t-\tau)^{\frac{1}{2} (\frac{3}{p_2} - \frac{3}{p_1}-1)} \norm{F(\cdot,\tau)}_{L^{p_1}(\R^3)} \, d\tau \leq $$ $$
		\leq c \int_0^t (t-\tau)^{\frac{1}{2} (\frac{3}{p_2} - \frac{3}{p_1}-1)} \tau^{\frac{s_1}{2}} \, d\tau \times \sup_{0 < \tau < T} \tau^{-\frac{s_1}{2}} \norm{F(\cdot,\tau)}_{L^{p_1}(\R^3)} \leq $$ $$
		\stackrel{\eqref{eq:Oseenconds}}{\leq} c \Big[ \big( \frac{s_1}{2} + 1 \big)^{-1} - 2 \big(\frac{3}{p_2} - \frac{3}{p_1} + 1 \big)^{-1} \Big] t^{\frac{1}{2} (\frac{3}{p_2} - \frac{3}{p_1} + s_1 + 1)} \times \norm{F}_{\cK^{s_1}_p(Q_T)} \leq $$ $$
		\stackrel{\eqref{eq:Oseenscaling}}{\leq} c \Big[ \big( \frac{s_1}{2} + 1 \big)^{-1} - 2 \big(\frac{3}{p_2} - \frac{3}{p_1} + 1 \big)^{-1} \Big] t^{\frac{s_2}{2}} \times \norm{F}_{\cK^{s_1}_p(Q_T)}.
	\end{equation}
This completes the proof of the first estimate. Now let us denote
	\begin{equation}
		u(\cdot,t) := \int_0^t e^{(t-s)} \bP \div F(s) \,ds
		\label{}
	\end{equation}
	for all $0 < t \leq T$
	and observe the identity
	\begin{equation}\label{eq:continuityid}
		u(\cdot,t) = e^{(t-s)\Delta} u(\cdot,\tau) + \int_{s}^t e^{(t-\tau)\Delta} \bP \div F(\cdot,\tau) \, d\tau
	\end{equation}
	for all $0 < s < t$. To prove that $u \in C((0,T];L^{p_2}(\R^3))$, one merely estimates
\begin{equation}\label{eq:continuityproof}
		\norm{u(\cdot,t) - u(\cdot,s)}_{L^{p_2}(\R^3)} \leq $$ $$ \leq \norm{e^{(t-s)\Delta} u(\cdot,s) - u(\cdot,s)}_{L^{p_2}(\R^3)} + \int_s^t \norm{e^{(t-\tau) \Delta} \bP \div F(\cdot,\tau)}_{L^{p_2}(\R^3)} \, d\tau \leq $$ $$
		\leq o(1) + c \int_s^t (t-\tau)^{\frac{1}{2}(\frac{3}{p_2} - \frac{3}{p_1} - 1)} \tau^{\frac{s_1}{2}} \, d\tau \times \norm{F}_{\cK^{s_1}_{p_1}(Q_T)} = o(1)
	\end{equation}
	as $|t-s| \to 0$, according to the assumption \eqref{eq:Oseenconds} on the exponents. 
	
	Let us now demonstrate how to prove the estimates on spatial derivatives. One estimates the integral in two parts,
	\begin{equation}
		\norm{ \int_0^t \nabla^l e^{(t-\tau) \Delta} \bP \div F(\cdot,\tau) \, d\tau}_{L^{p_2}(\R^3)} \leq $$ $$\leq c(l) \int_0^{\frac{t}{2}} (t-\tau)^{\frac{1}{2}(\frac{3}{p_2}-\frac{3}{p_1}-1-l)} \norm {F(\cdot,\tau)}_{L^{p_1}(\R^3)} \, d\tau +$$ $$+ c \int_{\frac{t}{2}}^t (t-\tau)^{\frac{1}{2} (\frac{3}{p_2}-\frac{3}{p_1}-1)} \norm {\nabla^l F(\cdot,\tau)}_{L^{p_1}(\R^3)} \, d\tau \leq$$ $$
		\leq c(l) \int_0^{\frac{t}{2}} (t-\tau)^{\frac{1}{2}(\frac{3}{p_2}-\frac{3}{p_1}-1-l)} \tau^{\frac{s_1}{2}} \, d\tau \times \norm{F}_{\cK^{s_1}_{p_1}(Q_T)} +$$ $$ c \int_{\frac{t}{2}}^t (t-\tau)^{\frac{1}{2} (\frac{3}{p_2}-\frac{3}{p_1}-1)} \tau^{\frac{1}{2}(s_1-l)} \, d\tau \times \norm{\tau^{\frac{l}{2}} \nabla^l F(\cdot,\tau)}_{\cK^{s_1}_{p_1}(Q_T)} \leq $$ $$
		\leq c(l,s_1,p_1,p_2) t^{\frac{1}{2} (s_2-l)} (\norm{F}_{\cK^{s_1}_{p_1}(Q_T)} + \norm{\tau^{\frac{l}{2}} \nabla^l F(\cdot,\tau)}_{\cK^{s_1}_{p_1}(Q_T)}).
	\end{equation}
	The proof of continuity in $L^{p_2}(\R^3)$ is similar to \eqref{eq:continuityproof} except with spatial derivatives in the identity \eqref{eq:continuityid}.

	The proof of estimates on the temporal derivatives is slightly more cumbersome due to the weighted spaces under consideration and that the temporal derivatives do not preserve the form of the equation. 
	By differentiating the identity \eqref{eq:continuityid} in time, one obtains
	\begin{equation}
		\p_t u(\cdot,t) = \p_t e^{(t-s)\Delta} u(\cdot,s) + $$ $$+ e^{(t-s)\Delta} \bP \div F(\cdot,s) + \int_s^t e^{(t-\tau)\Delta} \bP \div \p_\tau F(\cdot,\tau) \, d\tau,
	\end{equation}
	and more generally,
	\begin{equation}\label{eq:moregenerally}
		\p_t^k u(\cdot,t) = \p_t^k e^{(t-s)\Delta} u(\cdot,\tau) + $$ $$+ \sum_{\alpha=1}^k \p_t^{k-\alpha} e^{(t-s)\Delta} \bP \div \p_s^{\alpha-1} F(\cdot,s) + \int_s^t e^{(t-\tau)\Delta} \bP \div \p_\tau^k F(\cdot,\tau) \, d\tau.
	\end{equation}
	(In obtaining the identities, it is beneficial to compare with the differential form of the equation.) Now set $s := t/2$ and denote the terms by $I$, $II$, and $III$, respectively. We estimate
	\begin{equation}
		\norm{I}_{L^{p_2}} \leq c(k) t^{-k} \norm{u(\cdot,t/2)}_{L^{p_2}} \leq c(k,p_2) t^{-k+\frac{s_2}{2}} \norm{u}_{\cK^{s_2}_{p_2}(Q_T)} \leq $$ $$ \leq c(k,s_1,p_1,p_2) t^{-k+\frac{s_2}{2}} \norm{F}_{\cK^{s_1}_{p_1}(Q_T)},
	\end{equation}
	according to our original estimate. Furthermore,
		\begin{equation}
			\norm{II}_{L^{p_2}} \leq c(k) \sum_{\alpha=1}^k t^{\alpha-k+\frac{1}{2}(\frac{3}{p_2} - \frac{3}{p_1}-1)} \norm{(\p_t^{\alpha-1} F)(t/2)}_{L^{p_1}} \leq $$ $$ \leq c(k,s_1,p_1,p_2) t^{-k+\frac{s_2}{2}} \sum_{\alpha=1}^k \norm{\tau^{\alpha-1} F}_{\cK^{s_1}_{p_1}(Q_T)},
	\end{equation}
	and finally,
		\begin{equation}
			\norm{III}_{L^{p_2}} \leq c \int_{\frac{t}{2}}^t (t-\tau)^{\frac{1}{2}(\frac{3}{p_2} - \frac{3}{p_1} - 1)} \norm{\p_\tau^k F(\cdot,\tau)}_{L^{p_1}} \, ds \leq $$ $$ \leq c \int_{\frac{t}{2}}^t (t-\tau)^{\frac{1}{2}(\frac{3}{p_2} - \frac{3}{p_1} - 1)} \tau^{-k+\frac{s_1}{2}} \, ds \times \norm{\tau^k \p_\tau^k F}_{\cK^{s_1}_{p_2}(Q_T)} \leq $$ $$
			\leq c(k,s_1,p_1,p_2) t^{-k+\frac{s_2}{2}} \norm{\tau^k \p_\tau^k F}_{\cK^{s_1}_{p_2}(Q_T)}.
	\end{equation}
	This completes the proof of the time-derivative estimates. The proof of continuity is similar to \eqref{eq:continuityproof} except that one must use the identity \eqref{eq:moregenerally}.
	
 	Regularity in spacetime may be obtained by applying the temporal estimates to the spatial derivatives, since the spatial derivatives preserve the form of the equation.
\end{proof}

\subsection{Mild solutions in homogeneous Besov spaces}
The goal of this subsection is to review the well-posedness theory of the Navier-Stokes equations with initial data in Besov spaces.

Let us recall the notion of a mild solution to the Navier-Stokes equations, i.e., a tempered distribution $u$ on spacetime that satisfies the integral equations
\begin{equation}\label{eq:milddef}
	u(\cdot,t) = e^{t\Delta} u_0 - \int_0^t e^{(t-s)\Delta} \bP \div u \otimes u \,ds
\end{equation}
in a suitable function space.
The operator $e^{t\Delta}\bP \div$ is defined by convolution with the gradient of the Oseen kernel, see Chapter 11 of \cite{lemarie1}. We will often simply write
\begin{equation}
	u(\cdot,t) = e^{t\Delta} u_0 - B(u,u)(\cdot,t),
	\label{}
\end{equation}
\begin{equation}
	B(v,w)(\cdot,t) := \int_0^t e^{(t-s)\Delta} \bP \div v \otimes w \, ds.
	\label{}
\end{equation}
Distributional solutions to the Navier-Stokes equations are mild under rather general hypotheses, as discussed in Chapter 14 of \cite{lemarie1}. Small-data-global-existence in the spirit of Kato's seminal work \cite{kato} is known for divergence-free initial data in the following spaces:
\begin{equation}\label{chain}
	\dot H^{1/2}(\R^3) \into L^3(\R^3) \into \dot B^{s_{p_1}}_{p_1,q_1}(\R^3) \into $$ $$ B^{s_{p_2}}_{p_2,q_2}(\R^3)
\into BMO^{-1}(\R^3),
\end{equation}
where $3<p_1\leq p_2 <\infty$ and $3<q_1 \leq q_2 \leq\infty$. The case $BMO^{-1}(\R^3)$ was treated in \cite{tataru} and appears to be optimal. Ill-posedness has been demonstrated in \cite{bourgain} in the maximal critical space $\dot B^{-1}_{\infty,\infty}(\R^3)$. Local-in-time existence is known for initial data in the spaces
\begin{equation}
	\dot H^{1/2}(\R^3) \into L^3(\R^3) \into \dot B^{s_p}_{p,q}(\R^3) \into VMO^{-1}(\R^3)
\end{equation}
as long as $3 < p,q < \infty$, where 
\begin{equation}
	VMO^{-1}(\R^3) := \overline{\cS(\R^3)}^{BMO^{-1}(\R^3)}.
	\label{}
\end{equation}
The existence theory for initial data in homogeneous Besov spaces is summarized in the following theorem.

\begin{thm}[Mild solutions in critical Besov spaces]\label{mildexist}
	Let $3<p,q<\infty$ and $u_0 \in \dot B^{s_p}_{p,q}(\R^3)$ be a divergence-free vector field. Then there exists a time $0 < T^*(u_0) \leq \infty$ and a mild solution $u$ of the Navier-Stokes equations such that
	\begin{equation}\label{eq:class1}
		u \in C([0,T];\dot B^{s_p}_{p,q}(\R^3)) \cap \tilde L^1_T \dot B^{s_p + 2}_{p,q} \cap \tilde L^\infty_T \dot B^{s_p}_{p,q},
	\end{equation}
	\begin{equation}\label{eq:class2}
	u \in \mathring{\cK}_p(Q_T) \cap \mathring{\cK}_\infty(Q_T) \cap C((0,T];L^p \cap L^\infty(\R^3)),
	\end{equation}
	for all $0 < T < T^*(u_0)$. In addition, $u$ is the unique mild solution satisfying \eqref{eq:class1} and the unique mild solution satisfying $\eqref{eq:class2}$. Lastly, if $T^*(u_0) < \infty$, then the following two conditions are satisfied:
	\begin{equation}
		\text{(i) for all $p_0 \in [p,\infty]$, } \lim_{t \upto T^*} \norm{u(\cdot,t)}_{L^{p_0}(\R^3)} = \infty,
	\end{equation}
	\begin{equation}
		\text{ (ii) for all $p_0 \in [p,\infty)$, $q_0 \in [q,\infty)$, and $1 \leq r_0 \leq \infty$ such that $s_{p_0}+\frac{2}{r_0} \in (0,\frac{3}{p_0})$, } $$ $$
			\lim_{T \upto T^*} \norm{u}_{\tilde{L}^{r_0}_T \dot B^{s_{p_0}+\frac{2}{r_0}}_{p_0,q_0}} = \infty.
	\end{equation}
\end{thm}

\noindent In the statement, we have utilized the Banach space
\begin{equation}
	\mathring{\cK}^s_{p}(Q_T) := \left\lbrace u \in \cK^s_{p}(Q_T) : t^{-\frac{s}{2}} \norm{u(\cdot,t)}_{L^p(\R^3)} \to 0 \text{ as } t \dto 0 \right\rbrace.
	\label{}
\end{equation}

We will discuss the proof after reviewing two lemmas concerning quadratic equations in Banach spaces based on Lemma A.1 and A.2 in \cite{gallagher}. (See also Lemma 5 in \cite{auscher}.)
\begin{lem}[Abstract Picard lemma]\label{abstract}
	Let $X$ be a Banach space, $L \: X \to X$ a bounded linear operator such that $I-L \: X \to X$ is invertible, and
	$B$ a continuous bilinear operator on $X$ satisfying
	\begin{equation}
		\norm{B(x,y)}_X \leq \gamma \norm{x}_X \norm{y}_X
	\end{equation}
	for some $\gamma > 0$ and all $x$, $y \in X$. Then for all $a \in X$ satisfying
	\begin{equation}\label{eq:smallnesscond}
		\norm{(I-L)^{-1}a}_X < \frac{1}{4\norm{(I-L)^{-1}}_X \gamma},
	\end{equation}
	the Picard iterates $P_k(a)$, defined recursively by
	\begin{equation}
		P_0(a) := a, \quad P_{k+1}(a) := a + L(P_k) + B(P_k,P_k), \, k \geq 0,
	\end{equation}
	converge in $X$ to the unique solution $x \in X$ of the equation
	\begin{equation}\label{eq:abstracteq}
		x = a + L(x) + B(x,x)
	\end{equation}
	such that
	\begin{equation}
		\norm{x}_X < \frac{1}{2 \norm{(I-L)^{-1}}_X \gamma}.
	\end{equation}
\end{lem}

\noindent Regarding the hypothesis on $L$, the operator $I-L \: X \to X$ is invertible with norm 
\begin{equation}
	\norm{(I-L)^{-1}}_X \leq \frac{1}{1-\norm{L}_X}
	\label{}
\end{equation}
whenever $\norm{L}_X < 1$. We use this fact in Lemma \ref{stokes} and Theorem \ref{mildcalderon}.

Often one applies Lemma \ref{abstract} to an intersection of spaces. For instance, to solve the Navier-Stokes equations with divergence-free initial data $u_0 \in L^2 \cap L^3(\R^3)$, the space $X$ may be chosen as $X := L^5_{t,x} \cap L^\infty_t L^2_x \cap L^2_t \dot H^1_x(Q_T)$. Similarly, one may choose $X$ to include higher derivatives in order to prove higher regularity. Technically, when one includes more derivatives in the space $X$, one may need to shorten the time interval on which Lemma \ref{abstract} is applied and argue that the additional regularity is propagated forward in time. This is cleverly avoided in \cite{pavlovic}.

\begin{lem}[Propagation of regularity]\label{propagation}
	In the notation of Lemma \ref{abstract}, let $E \into X$ be a Banach space. Suppose that $L$ is bounded on $E$ such that $I-L \: E \to E$ is invertible and $B$ maps $E \times X \to E$ and $X \times E \to E$ with
	\begin{equation}
			\max(\norm{B(y,z)}_E,\norm{B(z,y)}_E) \leq \eta \norm{y}_E \norm{z}_X
	\end{equation}
	for some $\eta > 0$ and all $y \in E$, $z \in X$. Finally, suppose that \begin{equation}\label{eq:constantscond}
		\norm{(I-L)^{-1}}_{E} \eta \leq \norm{(I-L)^{-1}}_{X} \gamma.
	\end{equation}
	For all $a \in E$ satisfying \eqref{eq:smallnesscond}, the solution $x$ from Lemma \ref{abstract} belongs to $E$ and satisfies
	\begin{equation}
		\norm{x}_E \leq 2 \norm{(I-L)^{-1} a}_E.
	\end{equation}
\end{lem}
\noindent Lemma \ref{propagation} does not require the quantity $\norm{a}_E$ to be small, but one may have to increase $\gamma > 0$ in order to meet the condition \eqref{eq:constantscond}.

\begin{proof}[Proof of Theorem \ref{mildexist}]
		Let us assume the hypotheses of Theorem \ref{mildexist}. Let $r > 2$ such that $s_p + \frac{2}{r} \in (0,\frac{3}{p})$.

	1. \emph{Constructing a local-in-time mild solution}. To obtain local-in-time solutions, we will apply Lemma \ref{abstract} to the integral formulation \eqref{eq:milddef} of the Navier-Stokes equations in the Banach space
	\begin{equation}
		X_{T} := \tilde{L}^r_{T} \dot B^{s_p+\frac{2}{r}}_{p,q} \cap\mathring{\cK}_p(Q_{T}).
		\label{XTdef}
	\end{equation}
	Note that the realization condition in \eqref{eq:timespacedef} is satisfied, so the time-space homogeneous Besov space in \eqref{XTdef} is complete.
	
	Let us prove that the bilinear operator $B$ is bounded on $X_T$. In fact, it is bounded separately on the two spaces in the intersection. To prove that $B$ is bounded on $\cK_p(Q_T)$, we use H{\"o}lder's inequality,
	\begin{equation}
		\norm{u \otimes v}_{\cK^{2s_p}_{\frac{p}{2}}(Q_T)} \leq \norm{u}_{\cK_p(Q_T)} \norm{v}_{\cK_p(Q_T)},
		\label{}
	\end{equation}
	and conclude with the heat estimates in Lemma \ref{katoest}.
	That the subspace $\mathring{\cK}_p(Q_T)$ is stabilized follows from taking the limit $T \dto 0$.
	To prove boundedness on $\tilde{L}^r_T \dot B^{s_p+\frac{2}{r}}_{p,q}$, we use the Bony decomposition \eqref{eq:paraproduct}. First, apply the low-high paraproduct estimate \eqref{eq:lowhighest} and Sobolev embedding to obtain
	\begin{equation}\label{eq:paraproductest1}
		\norm{\dot T_u v}_{\tilde{L}^{\frac{r}{2}}_T \dot B^{-1+s_p+\frac{4}{r}}_{p,\frac{q}{2}}} \leq c \norm{u}_{\tilde{L}^r_T \dot B^{-1+\frac{2}{r}}_{\infty,q}} \norm{v}_{\tilde{L}^r_T \dot B^{s_p+\frac{2}{r}}_{p,q}} \leq $$ $$ \leq c \norm{u}_{\tilde{L}^r_T \dot B^{s_p+\frac{2}{r}}_{p,q}} \norm{v}_{\tilde{L}^r_T \dot B^{s_p+\frac{2}{r}}_{p,q}}.
	\end{equation}
	The analogous estimate is valid for $\dot T_v u$. Now we combine the heat estimate \eqref{eq:heatestbesovf}, substituting $F := \dot T_u v + \dot T_v u$, with the low-high paraproduct estimate:
	\begin{equation}
		\norm{\int_0^t e^{(t-s) \Delta} \bP \div (\dot T_u v + \dot T_v u) \, ds}_{\tilde{L}^r_T \dot B^{s_p+\frac{2}{r}}_{p,q}} \leq c \norm{u}_{\tilde{L}^r_T \dot B^{s_p+\frac{2}{r}}_{p,q}} \norm{v}_{\tilde{L}^r_T \dot B^{s_p+\frac{2}{r}}_{p,q}}.
		\label{}
	\end{equation}
	To estimate the high-high contribution, we apply the property \eqref{eq:highhighest} to obtain
	\begin{equation}\label{eq:paraproductest2}
		\norm{\dot R(u,v)}_{\tilde{L}^{\frac{r}{2}}_T \dot B^{2(s_{p}+\frac{2}{r})}_{\frac{p}{2},\frac{q}{2}}} \leq c\norm{u}_{\tilde{L}^r_T \dot B^{s_{p}+\frac{2}{r}}_{p,q}} \norm{v}_{\tilde{L}^r_T \dot B^{s_{p}+\frac{2}{r}}_{p,q}}.
	\end{equation}
	Notice that $2(s_{p}+\frac{2}{r}) \in (0,\frac{6}{p})$, so the realization condition in \eqref{eq:timespacedef} is satisfied. We then apply the heat estimate \eqref{eq:heatestbesovf} with $F := \dot R(u,v)$ to obtain
		\begin{equation}
		\norm{\int_0^t e^{(t-s) \Delta} \bP \div \dot R(u,v) \, ds}_{\tilde{L}^r_T \dot B^{s_p+\frac{2}{r}}_{p,q}} \leq c \norm{u}_{\tilde{L}^r_T \dot B^{s_p+\frac{2}{r}}_{p,q}} \norm{v}_{\tilde{L}^r_T \dot B^{s_p+\frac{2}{r}}_{p,q}}.
	\end{equation}
	This completes the proof of boundedness in $\tilde{L}^r_T \dot B^{s_p+\frac{2}{r}}_{p,q}$.
	
	The last step in applying Lemma \ref{abstract} is to obtain the smallness condition \eqref{eq:smallnesscond}. Since $u_0 \in \dot B^{s_p}_{p,q}(\R^3)$, we have from \eqref{katoendup} and \eqref{eq:heatestbesovu0} that
	\begin{equation}
		e^{t\Delta} u_0 \in X_\infty, \quad \norm{e^{t\Delta} u_0}_{X_T} \to 0 \text{ as } T \dto 0.
		\label{}
	\end{equation}
	Hence, \eqref{eq:smallnesscond} will be satisfied as long as $0 < T \ll 1$ depending on $u_0$. This completes Step 1.

	2. \emph{Further properties of mild solutions}. Our next goal is to prove that a given mild solution $u \in X_T$ belongs to the full range of function spaces stated in the theorem. It is clear from \eqref{katoendup} and \eqref{eq:heatestbesovu0} that $e^{t \Delta} u_0$ is in the desired spaces, so we will focus on the mapping properties of the nonlinear term.
	
	First, suppose $u$ is a mild solution in $\mathring{\cK}_p(Q_T)$. The integral equation \eqref{eq:milddef} combined with the estimates in Lemma \ref{katoest} allow one to bootstrap $u$ into the space $\mathring{\cK}_\infty(Q_T)$. The same estimates in Lemma \ref{katoest} also give $u \in C((0,T]; L^p \cap L^\infty(\R^3))$. We conclude that $u$ belongs to \eqref{eq:class2}.
	
	Now suppose $u$ is a mild solution in $\tilde{L}^r_T \dot B^{s_p+\frac{2}{r}}_{p,q}$. We may combine the paraproduct estimates \eqref{eq:paraproductest1} and \eqref{eq:paraproductest2} in Step 1 with the heat estimates \eqref{eq:heatestbesovf} to obtain the mapping property
	\begin{equation}\label{eq:Bmappingbesov}
		B \: \tilde{L}^r_T \dot B^{s_p+\frac{2}{r}}_{p,q} \times \tilde{L}^r_T \dot B^{s_p+\frac{2}{r}}_{p,q} \to C([0,T];\dot B^{s_p}_{p,q}(\R^3)) \cap \tilde{L}^\infty_T \dot B^{s_p}_{p,q}.
	\end{equation}
	Demonstrating $u \in \tilde{L}^1_T \dot B^{s_p+2}_{p,q}$ requires a bootstrapping argument that we borrow from Remark A.2 of \cite{gallagher}. Specifically, suppose that $u \in \tilde{L}^{r_0}_T \dot B^{s_p+\frac{2}{r_0}}_{p,q}$ for some $r_0 > 1$ such that $s_p + \frac{2}{r_0} > 0$. Let us define the exponents
	\begin{equation}
		\frac{1}{l(\varepsilon)} := \frac{1-\varepsilon}{2} - \frac{3}{2p}, \quad \frac{1}{r_1(\varepsilon)} := \frac{1}{r_0} + \frac{1}{l(\varepsilon)}, \quad 0 \leq \varepsilon < s_p + \frac{2}{r_0}.
		\label{}
	\end{equation}
	From interpolation, it is clear that $u \in \tilde{L}^{l(\varepsilon)}_T \dot B^{s_p+\frac{2}{l(\varepsilon)}}_{p,q}$. Now consider the additional restrictions $\varepsilon > 0$ and $r_1(\varepsilon) \geq 1$.
	One may verify that
	\begin{equation}
		s_p + \frac{2}{l(\varepsilon)} = -\varepsilon, \quad s_1(\varepsilon) := 2s_p + \frac{2}{r_0} + \frac{2}{l(\varepsilon)} = s_p + \frac{2}{r_0} - \varepsilon > 0.
		\label{}
	\end{equation}
	According to the paraproduct laws \eqref{eq:lowhighest} and \eqref{eq:highhighest}, we have
	\begin{equation}
		\norm{\dot T_u u}_{\tilde{L}^{r_1(\varepsilon)}_T \dot B^{s_1(\varepsilon)}_{\frac{p}{2},\frac{q}{2}}} + \norm{\dot R(u,u)}_{\tilde{L}^{r_1(\varepsilon)}_T \dot B^{s_1(\varepsilon)}_{\frac{p}{2},\frac{q}{2}}} \leq $$ $$
		\leq c \norm{u}_{\tilde{L}^{l(\varepsilon)}_T \dot B^{-\varepsilon}_{p,q}} \norm{u}_{\tilde{L}^{r_0}_T \dot B^{s_p+\frac{2}{r_0}}_{p,q}}.
	\end{equation}
	We then apply the heat estimate \eqref{eq:heatestbesovf} to obtain
	\begin{equation}
		B(u,u) \in \tilde{L}^{r_1(\varepsilon)}_T \dot B^{s_p+\frac{2}{r_1(\varepsilon)}}_{p,q},
		\label{}
	\end{equation}
	provided that $0 < \varepsilon < s_p + \frac{2}{r_0}$ and $r_1(\varepsilon) \geq 1$.
	In fact, by repeating the arguments with a new value $1 < \tilde{r}_0 < r_0$, one may treat the case $\varepsilon = 0$. The final result is that
	\begin{equation}
		B(u,u) \in \tilde{L}^{\max(1,r_1)}_T \dot B^{s_p+\min(2,\frac{2}{r_1})}_{p,q}, \quad \frac{1}{r_0} - \frac{s_p}{2} = \frac{1}{r_1}.
		\label{}
	\end{equation}
	Hence, one may improve $1/r_0$ by the fixed amount $-s_p/2$ at each iteration. This completes the bootstrapping argument, so $u$ belongs to the class \eqref{eq:class1}.

	3. \emph{Uniqueness}.
	Suppose that $u_1, u_2$ are two mild solutions on $Q_T$ in the class \eqref{eq:class1} with the same initial data $u_0$. For contradiction, assume there exists a time $0 < T_0 < T$ such that $u_1 \equiv u_2$ on $Q_{T_0}$ but $u_1$ and $u_2$ are not identical on $Q_{T_0,T_0+\delta}$ for all $0 < \delta < T-T_0$. Because the solutions are continuous with values in $\dot B^{s_p}_{p,q}(\R^3)$, we must have $\tilde{u_0} := u_1(\cdot,T_0) = u_2(\cdot,T_0)$. When $0 < \delta \ll 1$, depending on $u_1$ and $u_2$, the two solutions on $Q_{T_0,T_0+\delta}$ with the same initial data $\tilde{u_0}$ fall into the perturbative regime of Lemma \ref{abstract} in the critical time-space homogeneous Besov spaces. Hence, the solutions coincide on $Q_{T_0,T_0+\delta}$, which is a contradiction.

	Similarly, suppose that $u_1, u_2$ are two mild solutions on $Q_T$ in the class \eqref{eq:class2}. For $0 < \delta \ll 1$ depending on $u_1$ and $u_2$, the solutions on $Q_{\delta}$ fall into the perturbative regime of Lemma \ref{abstract} in the Kato space $\mathring{\cK}_p(Q_\delta)$. Hence, the solutions coincide on a small time interval. Uniqueness may be propagated forward in time according to the subcritical theory in $L^p(\R^3)$.

	4. \emph{Characterizing the maximal time of existence}.
	We now return to the mild solution $u$ that we constructed in Step 1. The solution may be continued according to the subcritical theory in $L^p(\R^3)$ and the critical theory in time-space homogeneous Besov spaces. The result is the following. There exists a time $0 < T^*(u_0) \leq \infty$ such that for all $0 < T < T^*(u_0)$, $u$ is the unique mild solution in $X_T$ with initial data $u_0$, and for all $T > T^*(u_0)$, the solution $u$ cannot be extended in $X_T$. The time $T^*(u_0)$ is the \emph{maximal time of existence} of the mild solution $u$ with initial data $u_0$. In the proof, we will denote it simply by $T^*$.

	Suppose that $T^* < \infty$. Then $\lim_{T \upto T^*} \norm{u}_{X_T} = \infty$,
	since otherwise the solution can be continued past $T^*$. \emph{A priori}, we know that
	\begin{equation}
		\text{(i') } \lim_{t \upto T^*} \norm{u(\cdot,t)}_{L^p(\R^3)} = \infty	\quad \text{ or } \quad \text{(ii') } \lim_{T \upto T^*} \norm{u}_{\tilde{L}^r_T \dot B^{s_p+\frac{2}{r}}_{p,q}} = \infty.
		\label{}
	\end{equation}
	Note that we may avoid writing $\limsup_{t \upto T^*} \norm{u(\cdot,t)}_{L^p(\R^3)}$ in (i') because the $L^p(\R^3)$ norm is subcritical. It is immediate that (i) implies (i') and (ii) implies (ii'). We will now demonstrate the implications in the reverse direction.
	
Suppose that (i) does not hold. In other words, there exists $p \leq p_0 \leq \infty$ such that $u \in \cK_{p_0}(Q_{T^*})$. By the boostrapping in Kato spaces mentioned in Step 2, we obtain that \begin{equation}
u \in \cK_\infty(Q_{T^*}) \cap C((0,T^*];L^\infty(\R^3)).
\label{}
\end{equation} Then, we may apply Lemma \ref{propagation} concerning propagation of regularity in the spaces $X = L^\infty(Q_T)$ and $E = X \cap L^\infty_t L^p_x(Q_T)$ with initial data $u(\cdot,t_0) \in L^p \cap L^\infty(\R^3)$ and initial time $t_0$ close to $T^*$. The lemma prevents $\norm{u(\cdot,t)}_{L^p(\R^3)}$ from blowing up as $t \upto T^*$, so (i') does not hold. This is the same argument as in the last step of Proposition \ref{singularityforms}.

Now suppose that (ii) does not hold. Then there exists $p \leq p_0 < \infty$, $q \leq q_0 < \infty$, and $r_0 > 2$ such that $s_{p_0}+\frac{2}{r_0} \in (0,\frac{3}{p_0})$ and $u \in \tilde{L}^{r_0}_{T^*} \dot B^{s_{p_0}+\frac{2}{r_0}}_{p_0,q_0}$. By the arguments in Step 2, we must have that
\begin{equation}
u \in C((0,T^*]; \dot B^{s_{p_0}}_{p_0,q_0}) \cap \tilde{L}^{1}_{T^*} \dot B^{s_{p_0}+2}_{p_0,q_0} \cap \tilde{L}^\infty_{T^*} \dot B^{s_{p_0}}_{p_0,q_0}.
	\label{}
\end{equation}
As in the previous paragraph, we may apply Lemma \ref{propagation} in the spaces $X = \tilde{L}^{r_0}_T \dot B^{s_{p_0} + \frac{2}{r_0}}_{p_0,q_0}$, $E = X \cap \tilde{L}^{r}_T \dot B^{s_p+\frac{2}{r}}_{p,q}$, with initial data $u(\cdot,t_0) \in \dot B^{s_p}_{p,q} \cap \dot B^{s_{p_0}}_{p_0,q_0}(\R^3)$ and initial time $t_0$ close to $T^*$. This proves that the norm of $u$ in $\tilde{L}^r_{T^*} \dot B^{s_p+\frac{2}{r}}_{p,q}$ stays bounded, so (ii') fails. Here it is crucial that $u(\cdot,t)$ is continuous on $[0,T^*]$ with values in $\dot B^{s_{p_0}}_{p_0,q_0}(\R^3)$, so that the existence time of the solution with initial data $u(\cdot,t_0)$ has a uniform lower bound. This was not an issue in the subcritical setting. Also, we do not record here the bilinear estimates necessary to apply Lemma \ref{propagation}. They are similar to the estimates in Step 1. 

5. \emph{More characterizing}. It remains to prove that (i') is equivalent to (ii'). To begin, we will show that (ii) implies (i') by arguing the contrapositive. Suppose
	\begin{equation}\label{eq:ivimpliesilp}
		\sup_{T^*/4 \leq t \leq T^*} \norm{u(\cdot,t)}_{L^p(\R^3)} < \infty.
	\end{equation}
	By the subcritical theory in $L^p(\R^3)$, it is not difficult to show that
	\begin{equation}\label{eq:ivimpliesiw1p}
		\sup_{T^*/2 \leq t \leq T^*} \norm{\nabla u(\cdot,t)}_{L^p(\R^3)} < \infty.
	\end{equation}
	One may interpolate between \eqref{eq:ivimpliesilp} and \eqref{eq:ivimpliesiw1p} to obtain for all $0 < s < 1$,
	\begin{equation}
	\sup_{T^*/2 \leq t \leq T^*} \norm{u(\cdot,t)}_{\dot B^s_{p,1}(\R^3)} \leq $$ $$ \leq c(s) \sup_{T^*/2 \leq t \leq T^*} \norm{u(\cdot,t)}_{\dot B^1_{p,\infty}(\R^3)}^s \norm{u(\cdot,t)}_{\dot B^{0}_{p,\infty}(\R^3)}^{1-s} < \infty,
	\end{equation}
	see Proposition 2.22 in \cite{bahourichemin}. Now, let $m := \max(q,r)$.
	By Minkowski's and H{\"o}lder's inequalities,
	\begin{equation}
		\norm{u}_{\tilde{L}^r_{T^*/2,T^*} \dot B^{s_p+\frac{2}{r}}_{p,m}} \leq \norm{u}_{L^r_t (\dot B^{s_p+\frac{2}{r}}_{p,m})_x (Q_{T^*/2,T^*})} \leq $$ $$
		\leq c \, (T^*)^{\frac{1}{r}} \sup_{T^*/2 \leq t \leq T^*} \norm{u(\cdot,t)}_{\dot B^{s_p+\frac{2}{r}}_{p,1}(\R^3)} < \infty,
		\label{}
	\end{equation}
	and one concludes that (ii) fails.
	
	Now we will demonstrate that (i') implies (ii'), again by arguing the contrapositive. Let us assume that $\norm{u}_{\tilde{L}^r_{T^*} \dot B^{s_p+\frac{2}{r}}_{p,q}} < \infty$, so that by the mapping properties of the nonlinear term in Step 2, 
	\begin{equation}
		u \in C([0,T^*];\dot B^{s_p}_{p,q}(\R^3)).
		\label{eq:wemaysuppose}
	\end{equation}
	Our goal is to prove that the following quantity is bounded:
	\begin{equation}\label{eq:timebesovhigher}
		\norm{u}_{\tilde{L}^\infty_{T_0,T^*} \dot B^{\frac{3}{p}}_{p,q}(\R^3)} < \infty,
	\end{equation}
	for some $0 < T_0 < T^*$.
	Indeed, interpolating between \eqref{eq:wemaysuppose} and \eqref{eq:timebesovhigher}, one may demonstrate that
	\begin{equation}
		\sup_{T_0 \leq t \leq T^*} \norm{u(\cdot,t)}_{\dot B^{0}_{p,1}(\R^3)} \leq $$ $$ \leq c \sup_{T_0 \leq t \leq T^*} \norm{u(\cdot,t)}_{\dot B^{s_p}_{p,\infty}(\R^3)}^{s_p+1} \norm{u(\cdot,t)}_{\dot B^{s_p+1}_{p,\infty}(\R^3)}^{-s_p} < \infty.
	\end{equation}
	In light of the embedding $\dot B^0_{p,1}(\R^3) \into L^p(\R^3)$, this will complete the proof that (i') fails.
	
	To prove \eqref{eq:timebesovhigher}, we will need more estimates for the heat equation. Let $\cC \subset \R^3$ be an annulus and $\lambda > 0$.
	Let $0 < T \leq \infty$ and $f$ be a tempered distribution on $Q_T$ with spatial Fourier transform satisfying $\supp{\widehat{f}} \subset \lambda \cC \times [0,T]$. Then
	\begin{equation}
		\lambda \norm{t^{\frac{1}{2}} \int_0^t e^{(t-\tau)\Delta} f(\cdot,\tau) \, d\tau}_{L^{r_2}_t L^{p_2}_x(Q_T)} \leq $$ $$ \leq C \lambda^{2(\frac{1}{r_1} - \frac{1}{r_2} - 1)} \lambda^{3(\frac{1}{p_1}-\frac{1}{p_2})} (\norm{f}_{L^{r_1}_t L^{p_1}_x(Q_T)} + \lambda \norm{t^{\frac{1}{2}} f}_{L^{r_1}_t L^{p_1}_x(Q_T)}), $$ $$
	1 \leq p_1 \leq p_2 \leq \infty, \quad 1 \leq r_1 \leq r_2 \leq \infty.
\end{equation}
Recall the smoothing effect \eqref{eq:freqlocalizedheat} of the heat flow. When $p_1 = p_2$, one may write
\begin{equation}
	\lambda \norm{\int_0^t e^{(t-\tau)\Delta} f(\cdot,\tau) \, d\tau}_{L^{p_1}(\R^3)} \leq $$ $$ \leq C \int_0^{\frac{t}{2}} \norm{\nabla e^{(t-\tau)\Delta} f(\cdot,\tau)}_{L^{p_1}(\R^3)} \, d\tau + \lambda \int_{\frac{t}{2}}^t \norm{e^{(t-\tau) \Delta} f(\cdot,\tau)}_{L^{p_1}(\R^3)} \, d\tau \leq $$ $$ \leq C \int_{0}^{\frac{t}{2}} (t-\tau)^{-\frac{1}{2}} e^{-c(t-\tau)\lambda^2} \norm{f(\cdot,\tau)}_{L^{p_1}(\R^3)} \, d\tau + $$ $$ + C \lambda \int_{\frac{t}{2}}^t e^{-c(t-\tau)\lambda^2} \tau^{-\frac{1}{2}} \norm{\tau^{\frac{1}{2}} f(\cdot,\tau)}_{L^{p_1}(\R^3)} \,d\tau \leq $$ $$ \leq Ct^{-\frac{1}{2}} \big( \int_0^{\frac{t}{2}} e^{-c(t-\tau)\lambda^2} \norm{f(\cdot,\tau)}_{L^{p_1}(\R^3)} \, ds + \lambda \int_{\frac{t}{2}}^t e^{-c(t-\tau)\lambda^2} \norm{\tau^{\frac{1}{2}} f(\cdot,\tau)}_{L^{p_1}(\R^3)} \, d\tau \big)
\end{equation}
and treat the subsequent integration in time by Young's convolution inequality. The case $p_2 > p_1$ follows from Bernstein's inequality. This leads us to a higher regularity estimate in Besov spaces analogous to Lemma \ref{katoest}:
\begin{equation}\label{eq:newheatest}
	\norm{t^{\frac{1}{2}} \int_0^t e^{(t-s)\Delta} \bP \div F(\cdot,s) \, ds}_{\tilde{L}^{r_2}_T \dot B^{\frac{3}{p_2}}_{p_2,q_2}} \leq $$ $$ \leq c (\norm{F}_{\tilde{L}^{r_1}_T \dot B^{s_{p_1}+\frac{2}{r_1}-1}_{p_1,q_1}} + \norm{t^{\frac{1}{2}} F}_{\tilde{L}^{r_1}_T \dot B^{s_{p_1}+\frac{2}{r}}_{p_1,q_1}}),
\end{equation}
when also $1 \leq q_1 \leq q_2 \leq \infty$.

Let us return to the proof. We will apply Lemma \ref{abstract} in the following Banach spaces to obtain \eqref{eq:timebesovhigher}:
\begin{equation}
	Y_{T} := \big\lbrace v \in \tilde{L}^r_{T} \dot B^{s_p+\frac{2}{r}}_{p,q} : t^{\frac{1}{2}} v \in \tilde{L}^r_{T} \dot B^{\frac{3}{p} + \frac{2}{r}}_{p,q} \big\rbrace,
	\label{}
\end{equation}
\begin{equation}
	\norm{v}_{Y_{T}} := \max\big( \norm{v}_{\tilde{L}^r_{T} \dot B^{s_p+\frac{2}{r}}_{p,q}}, \norm{t^{\frac{1}{2}} v}_{\tilde{L}^r_{T} \dot B^{\frac{3}{p}+\frac{2}{r}}_{p,q}} \big).
	\label{}
\end{equation}
Let us show that $B$ is a bounded operator on $Y_T$ with norm $\kappa$ independent of the time $T > 0$. We need only prove boundedness of the upper part of the norm. Now, according to the low-high paraproduct estimate \eqref{eq:lowhighest} and Sobolev embedding,
\begin{equation}\label{newpara1}
		\norm{t^{\frac{1}{2}} \dot T_v w}_{\tilde{L}^{\frac{r}{2}}_T \dot B^{s_p+\frac{4}{r}}_{p,\frac{q}{2}}} \leq c \norm{v}_{\tilde{L}^r_T \dot B^{-1+\frac{2}{r}}_{\infty,q}} \norm{t^{\frac{1}{2}} w}_{\tilde{L}^r_T \dot B^{\frac{3}{p}+\frac{2}{r}}_{p,q}} \leq $$ $$ \leq c \norm{v}_{\tilde{L}^r_T \dot B^{s_p+\frac{2}{r}}_{p,q}} \norm{t^{\frac{1}{2}} w}_{\tilde{L}^r_T \dot B^{\frac{3}{p}+\frac{2}{r}}_{p,q}} \leq c \norm{v}_{Y_T} \norm{w}_{Y_T},
	\end{equation}
	and similarly for $\dot T_w v$.
	According to the high-high estimate \eqref{eq:highhighest}, we have
	\begin{equation}\label{newpara2}
		\norm{t^{\frac{1}{2}} \dot R(v,w)}_{\tilde{L}^{\frac{r}{2}}_T \dot B^{-1+\frac{6}{p}+\frac{4}{r}}_{\frac{p}{2},\frac{q}{2}}} \leq c\norm{v}_{\tilde{L}^r_T \dot B^{s_{p}+\frac{2}{r}}_{p,q}} \norm{t^{\frac{1}{2}} w}_{\tilde{L}^r_T \dot B^{\frac{3}{p}+\frac{2}{r}}_{p,q}} \leq $$ $$ \leq c\norm{v}_{Y_T} \norm{w}_{Y_T}.
	\end{equation}
	Since $v, w \in X_T$, there is no ambiguity modulo polynomials in forming the paraproduct operators. The bilinear estimate is then completed by combining \eqref{eq:newheatest} with (\ref{eq:paraproductest1}, \ref{newpara1}) for the low-high terms and (\ref{eq:paraproductest2}, \ref{newpara2}) for the high-high terms.
By the same paraproduct estimates and applying \eqref{eq:newheatest} with $r_2 = \infty$, one may show that
\begin{equation}
	B \: Y_T \times Y_T \to C((0,T];\dot B^{\frac{3}{p}}_{p,q}(\R^3)).
	\label{}
\end{equation}
This is the property we will use below.

We are ready to conclude. Recall that $u \in C([0,T^*];\dot B^{s_p}_{p,q}(\R^3))$ by \eqref{eq:wemaysuppose}. By continuity, there exists a time $T_1 > 0$ such that $\norm{e^{t\Delta} u(\cdot,t_0)}_{Y_{T_1}} < (4\kappa)^{-1}$ as $t_0 \upto T^*$. Now we may apply Lemma \ref{abstract} in the space $Y_{T_1}$ with $u(\cdot,T^*-T_1)$ as initial data, and the result agrees with the mild solution $u$ on the time interval $[T^*-T_1,T^*]$. Therefore, we may take $T_0 := T^*-T_1/2$ in \eqref{eq:timebesovhigher} to complete the proof. (In fact, by similar arguments of re-solving the equation, one may show that \eqref{eq:timebesovhigher} holds everywhere away from the initial time.)
\end{proof}

Let us also provide the following characterization of the maximal time of existence.

\begin{prop}[Formation of singularity at blow-up time]\label{singularityforms}
	Let $3 < p < \infty$ and $u_0 \in L^p(\R^3)$ be a divergence-free vector field. Let $T^*(u_0) > 0$ be the maximal time of existence of the unique mild solution $u$ such that $u(\cdot,t)$ is continuous in $L^p(\R^3)$. If $T^*(u_0) < \infty$, then $u$ must have a singular point at time $T^*(u_0)$.
\end{prop}
\begin{proof}
	This argument is based on similar arguments in \cite{rusin,jiasverak}. Let $u$ be as in the hypothesis of the proposition and assume that $T^*(u_0) < \infty$. Following Calder{\'o}n \cite{calderon}, for all $\delta > 0$, there exist divergence-free vector fields $U_0 \in L^2(\R^3) \cap L^p(\R^3)$ and $V_0 \in L^p(\R^3)$ such that
	\begin{equation}
		u_0 = U_0 + V_0, \quad \norm{V_0}_{L^p(\R^3)} < \delta.
		\label{}
	\end{equation}
	Let $V$ denote the unique mild solution associated to the initial data $V_0$, such that $V(\cdot,t)$ is continuous in $L^p(\R^3)$. Futhermore, one may choose $0 < \delta \ll 1$ such that $T^*(V_0) \geq 2 T^*(u_0)$. We will abbreviate $T^*(u_0)$ as $T^*$. The remainder $U$ solves the perturbed Navier-Stokes equations on $Q_{T^*}$,
	\begin{equation}\label{eq:Uequationapp}
		\p_t U - \Delta U + \div U \otimes U + \div U \otimes V + \div V \otimes U = - \nabla P, \quad
		\div U = 0,
	\end{equation}
	with the initial condition $U(\cdot,0) = U_0$. By the well-posedness theory for the equation \eqref{eq:Uequationapp} as well as the energy inequality, one may prove that $U$ is in the energy space up to the blow-up time:
	\begin{equation}
		U \in L^\infty_t L^2_x \cap L^2_t \dot H^1_x(Q_{T^*}).
		\label{}
	\end{equation}
	Based on the decomposition of $u$, we obtain
	\begin{equation}\label{eq:uregularity}
	u \in L^3(Q_{T^*}) + L^\infty_t L^p_x(Q_{T^*}), 
\end{equation}
\begin{equation}
	p \in L^{3/2}(Q_{T^*}) + L^\infty_t L^{p/2}_x(Q_{T^*}),
\end{equation}
where $p := (-\Delta)^{-1} \div \div u \otimes u$ is the pressure associated to $u$. Recall now that $u$ is in subcritical spaces, so one may justify the local energy inequality for $(u,p)$ up to the blow-up time $T^*$. Moreover, \eqref{eq:uregularity} implies
\begin{equation}
	\lim_{|x| \to \infty} \int_{0}^{T^*} \int_{B(x,1)} |u|^3 + |p|^{3/2} \, dx \,dt = 0.
	\label{}
\end{equation}
Therefore, by the $\varepsilon$-regularity criterion in Theorem \ref{ckncrit}, there exist constants $R, \kappa > 0$, and the set $K := (\R^3 \setminus B(R)) \times (T^*/2,T^*)$, such that
\begin{equation}\label{eq:controloutsidebigball}
	\sup_K |u(x,t)| < \kappa.
\end{equation}
Finally, suppose that $u$ has no singular points at time $T^*$. This assumption, paired with the estimate \eqref{eq:controloutsidebigball}, implies that $u \in L^\infty(Q_{\epsilon,T^*})$ for some $\epsilon \in (T^*/2, T^*)$.
Consider the bilinear estimates
\begin{equation}\label{bilinearshit}
	\norm{B(v,w)}_{L^\infty_t L^p_x(Q_T)} \leq cT^{1/2} \norm{v}_{L^\infty_t L^p_x(Q_T)} \norm{w}_{L^\infty(Q_T)} $$ $$
	\norm{B(v,w)}_{L^\infty_t L^p_x(Q_T)} \leq cT^{1/2} \norm{v}_{L^\infty(Q_T)} \norm{w}_{L^\infty_t L^p_x(Q_T)}
\end{equation}
for all $T>0$.
Now one applies Lemma \ref{propagation} with the bilinear estimates \eqref{bilinearshit} and the choice of spaces
\begin{equation}
	X := L^\infty(Q_{t_0,T^*}), \quad E := X \cap L^\infty_t L^p_x(Q_{t_0,T^*}),
	\label{}
\end{equation}
for each $\epsilon < t_0 < T^*$. This prevents $\norm{u(\cdot,t)}_{L^p(\R^3)}$ from becoming unbounded as $t \upto T^*$ and completes the proof.
\end{proof}
\begin{cor}\label{singularity}
	Let $u$ be the mild solution of Theorem \ref{mildexist} with initial data $u_0$. If $T^*(u_0) < \infty$, then $u$ has a singular point at time $T^*(u_0)$.
\end{cor}

Lastly, we record an existence theorem for mild solutions with initial data in subcritical Besov spaces. Similar results concerning smoothness of solutions belonging to the critical space $L^5(Q_T)$ were proven in \cite{dongdu}.

\begin{thm}[Mild solutions in subcritical Besov spaces]\label{subcrittheory}
	Let $3 < p,q \leq \infty$, $0 < \varepsilon < -s_p$, $s := s_p + \varepsilon$, and $M > 0$. There exists a positive constant $\kappa := \kappa(s,p)$
	such that for all divergence-free vector fields $u_0 \in \dot B^{s}_{p,q}(\R^3)$ with $\norm{u_0}_{\dot B^{s}_{p,q}(\R^3)} \leq M$, there exists a unique mild solution $u \in \cK^s_{p}(Q_{T})$ of the Navier-Stokes equations on $Q_{T}$ with initial data $u_0$ and
	$T^{\frac{\varepsilon}{2}} := \kappa/M$.
	Moreover, the solution satisfies
	\begin{equation}
	\p_t^k \nabla^l u \in C((0,T];L^p(\R^3)) \cap C((0,T];L^\infty(\R^3)),
		\label{}
	\end{equation}
	\begin{equation}\label{eq:subcritsatisfy}
		\norm{t^{k+\frac{l}{2}} \p_t^k \nabla^l u}_{\cK^{s}_{p}(Q_T)} + \norm{t^{k+\frac{l}{2}} \p_t^k \nabla^l u}_{\cK^{-1+\varepsilon}_{\infty}(Q_{T})} \leq c(s,p,k,l) M, 
	\end{equation}
	for all integers $k, l \geq 0$. Hence, $u$ is smooth in $Q_{T}$.
\end{thm}

\subsection{$\varepsilon$-regularity and backward uniqueness}

In this section, we record a number of important theorems relevant for the blow-up procedure in Theorem \ref{main}. Our primary reference is the seminal paper of Escauriaza, Seregin, and {\u S}ver{\'a}k \cite{sverak03}. 

We will now introduce the relevant notation and definitions. For $R > 0$ and $z_0 = (x_0,t_0) \in \R^{3+1}$, we define the Euclidean balls
\begin{equation}
	B(x_0,R) := \{ y \in \R^3 : |x_0-y| < R \}, \quad B(R) := B(0,R),
	\label{}
\end{equation}
and the parabolic balls
\begin{equation}
	Q(z_0,R) := B(x_0,R) \times (t_0-R^2,t_0), \quad Q(R) := Q(0,R).
	\label{}
\end{equation}

Suppose that $u \: Q(z_0,R) \to \R^3$ is a measurable function. We will say that $z_0$ is a \emph{singular point} of $u$ if for all $0 < r < R$, $u \not\in L^\infty(Q(z_0,r))$. In this case, we will say that $u$ is \emph{singular} at $z_0$. Otherwise, we will say that $z_0$ is a \emph{regular point} of $u$.

We say that $(v,q)$ is a \emph{suitable weak solution} of the Navier-Stokes equations in $Q(z_0,R)$ if the following requirements are satisfied:
\begin{enumerate}[(i)]
	\item $v \in L^\infty_t L^2_x \cap L^2_t H^1_x(Q(z_0,R))$ and $q \in L^{3/2}_{t,x}(Q(z_0,R))$,
	\item $(v,q)$ solves the Navier-Stokes equations in the sense of distributions on $Q(z_0,R)$,
	\item for all $0 \leq \varphi \in C^\infty_0(B(R) \times (t_0-R^2,t_0])$, $(v,q)$ satisfies the local energy inequality
		\begin{equation}\label{sveraksuitable}
			\int_{B(x_0,R)} \varphi |v(x,t)|^2 \, dx + 2 \int_{t_0-R^2}^t \int_{B(x_0,R)} \varphi |\nabla v|^2 \, dx \, dt' \leq $$ $$ \leq \int_{t_0-R^2}^t \int_{B(x_0,R)} (|v|^2 (\Delta \varphi + \p_t \varphi) + v \cdot \nabla\varphi (|v|^2 + 2q)) \, dx \, dt',
		\end{equation}
	for every $t \in (t_0-R^2,t_0]$.
\end{enumerate}
We say that $(v,q)$ is a suitable weak solution of the Navier-Stokes equations in an open set $\Omega \subset \R^{3+1}$ if $(v,q)$ is a suitable weak solution in each parabolic ball $Q \subset \Omega$.

Note that we permit the test functions $\varphi$ in the local energy inequality to be supported near $t_0$, but this is not strictly necessary. Also, one may use the Navier-Stokes equations to demonstrate that a suitable weak solution $u(\cdot,t)$ is weakly continuous on $[t_0-R^2,t_0]$ as a function with values in $L^2(B(x,R))$.

Let us show that the distributional local energy inequality \eqref{localenergyineq} in the definition of Calder{\'o}n solution implies the local energy inequality \eqref{sveraksuitable} in the definition of suitable weak solution. Let $0 \leq \eta \in C^\infty_0(\R)$ such that $\eta \equiv 1$ when $|\tau| \leq 1/4$, $\eta \equiv 0$ when $|\tau| \geq 1/2$, and $\int_\R \eta \, d\tau = 1$. For all $\varepsilon > 0$, define $\eta_\varepsilon(\tau) := \varepsilon^{-1} \eta (\varepsilon^{-1} \tau)$. Now let $0 \leq \varphi \in C^\infty_0(B(R) \times (t_0-R^2,t_0])$. Define
\begin{equation}
	\Phi_{t,\varepsilon}(x,t') := \varphi(x,t') \big(1-\int_{\infty}^{t'} \eta_\varepsilon(\tau-t) \, d\tau \big), \quad t \in (t_0-R^2,t_0), \; \varepsilon > 0.
	\label{}
\end{equation}
Using $\Phi_{t,\varepsilon}$ as a test function in \eqref{suitability} and passing to the limit $\varepsilon \dto 0$ proves \eqref{sveraksuitable} for almost every $t \in (t_0-R^2,t_0)$. That the inequality is true for all $t \in (t_0-R^2,t_0]$ follows from the weak continuity of $u(\cdot,t) \in L^2(B(x,R))$.

We will now state the Caffarelli-Kohn-Nirenberg $\varepsilon$-regularity criterion for suitable weak solutions, see \cite{CKN,lin,CKNSeregin,sverak03}.
\begin{thm}[$\varepsilon$-regularity criterion \cite{sverak03}]\label{ckncrit}
		There exist absolute constants $\varepsilon_0 > 0$ and $c_{0,k} > 0$, $k \in \N$, with the following property. Assume $(v,q)$ is a suitable weak solution on $Q(1)$ satisfying
		\begin{equation}
		\int_{Q(1)} |v|^3 + |q|^{3/2} \, dx \, dt < \varepsilon_0.
	\end{equation}
	Then, for each $k \in \N$, $\nabla^{k-1} v$ is H{\"o}lder continuous on $\overline{Q(1/2)}$ and satisfies
	\begin{equation}
		\sup_{Q(1/2)} |\nabla^{k-1} v(x,t)| < c_{0,k}.
	\end{equation}
\end{thm}

The $\varepsilon$-regularity criterion may be used to prove that singularities of suitable weak solutions persist under locally strong limits.
\begin{prop}[Persistence of singularity \cite{rusin}]\label{persist}
	Let $(v_k,q_k)$ be a sequence of suitable weak solutions on $Q(1)$ such that $v_k \to v$ in $L^3(Q(1))$, $q_k \wto q$ in $L^{3/2}(Q(1))$. Assume $v_k$ is singular at $(x_k,t_k) \in \overline{Q(1)}$ and $(x_k,t_k) \to 0$. Then $v$ is singular at the spacetime origin.
\end{prop}

Finally, we recall two theorems concerning backward uniqueness and unique continuation of solutions to differential inequalities.
\begin{thm}[Backward uniqueness \cite{sverak03}]\label{backuniqueness}
	Let $Q_+ := \R^3_+ \times (0,1)$. Suppose $u \: Q_+ \to \R^3$ satisfies the following conditions:
	\begin{enumerate}[(i)]
		\item $|u_t + \Delta u| \leq c (|\nabla u| + |u|)$ a.e. in $Q_+$ for some $c > 0$,
		\item $u(\cdot,0) = 0$,
		\item $|u(x,t)| \leq e^{M|x|^2}$ in $Q_+$, and
		\item $u, u_t, \nabla^2 u \in L^2_t (L^2_\loc)_x(Q_+)$.
	\end{enumerate}
	Then $u \equiv 0$ on $Q_+$.
\end{thm}

To make sense of condition (ii), one may use that $u \in C([0,1];\cD'(\R^3_+))$, due to condition (iv).
 
\begin{thm}[Unique continuation \cite{sverak03}]\label{continuation}
	Let $R, \, T > 0$ and $Q(R,T) := B(R) \times (0,T) \subset \R^{3+1}$. Suppose $u \: Q(R,T) \to \R^3$ satisfies the following conditions:
	\begin{enumerate}[(i)]
		\item $u, u_t, \nabla^2 u \in L^2(Q(R,T))$
		\item $|u_t + \Delta u| \leq c (|\nabla u| + |u|)$ a.e. in $Q(R,T)$ for some $c > 0$
		\item $|u(x,t)| \leq C_k (|x| + \sqrt{t})^k$ in $Q(R,T)$ for some $C_k > 0$, all $k \geq 0$
	\end{enumerate}
	Then $u(x,0) = 0$ for all $x \in B(R)$.
\end{thm}

\end{document}